\documentclass{article}

\usepackage{amsmath}
\usepackage{amssymb}
\usepackage{amsthm}
\usepackage{enumerate}
\usepackage{tikz-cd}
\usepackage{pgf,tikz}
\usetikzlibrary{arrows}
\usepackage{graphics}
\usepackage{commath}
\usepackage{epstopdf}
\usepackage{verbatim}
\usepackage{mathrsfs}
\usepackage{breqn}
\usepackage{authblk}
\usepackage{fancyhdr}

\usepackage[pagebackref=true,hypertexnames=false]{hyperref}

\title{Families of relatively exact Lagrangians, free loop spaces and generalised homology}
\author{Noah Porcelli}
\newcommand{\Addresses}{{% additional braces for segregating \footnotesize
  \bigskip
  \footnotesize

  \textsc{Imperial College London, South Kensington Campus, London SW7 2AZ, UK}\par\nopagebreak
  \textit{E-mail address}, \texttt{n.porcelli@imperial.ac.uk}

}}

\date{}
    
\usepackage{hyperref}
\usepackage{cleveref}
\newtheorem{TT}{Theorem}[section]
\newtheorem{LL}[TT]{Lemma}
\newtheorem{CC}[TT]{Corollary}
\newtheorem{PP}[TT]{Proposition}
\newtheorem{DD}[TT]{Definition}
\newtheorem{EE}[TT]{Example}
\newtheorem{RR}[TT]{Remark}
\newtheorem{QQ}[TT]{Question}
\newtheorem{BB}[TT]{Assumption}
\newtheorem{MM}[TT]{Claim}

\newcommand{\Ind}{\mathrm{Ind}}
\newcommand{\Ker}{\mathrm{Ker\,}}
\newcommand{\re}{\mathrm{Re\,}}
\newcommand{\im}{\mathrm{Im\,}}

\newcommand{\Rank}{\mathrm{Rank}}
\newcommand{\R}{\mathbb{R}}
\newcommand{\C}{\mathbb{C}}
\newcommand{\Q}{\mathbb{Q}}

\newcommand{\Z}{\mathbb{Z}}
\newcommand{\bP}{\mathbb{P}}
\newcommand{\Sphere}{\mathbb{S}}

\newcommand{\sP}{\mathcal{P}}
\newcommand{\sM}{\mathcal{M}}
\newcommand{\sN}{\mathcal{N}}

\newcommand{\sD}{\mathcal{D}}

\newcommand{\sW}{\mathcal{W}}
\newcommand{\sG}{\mathcal{G}}
\newcommand{\sE}{\mathcal{E}}
\newcommand{\sL}{\mathcal{L}}
\newcommand{\sU}{\mathcal{U}}

\newcommand{\sQ}{\mathcal{Q}}
\newcommand{\sR}{\mathcal{R}}

\newcommand{\Shrek}{\mathchar\numexpr"6000+`!\relax}
\begin{document}
    \maketitle
    \pagenumbering{arabic}
    \begin{abstract}
        We prove that (under appropriate orientation conditions, depending on $R$) a Hamiltonian isotopy $\psi^1$ of a symplectic manifold $(M, \omega)$ fixing a relatively exact Lagrangian $L$ setwise must act trivially on $R_*(L)$, where $R_*$ is some generalised homology theory. We use a strategy inspired by that of Hu, Lalonde and Leclercq (\cite{Hu-Lalonde-Leclercq}), who proved an analogous result over $\Z/2$ and over $\Z$ under stronger orientation assumptions. However the differences in our approaches let us deduce that if $L$ is a homotopy sphere, $\psi^1|_L$ is homotopic to the identity. Our technical set-up differs from both theirs and that of Cohen, Jones and Segal (\cite{CJS, C}).\par 
        We also prove (under similar conditions) that $\psi^1|_L$ acts trivially on $R_*(\sL L)$, where $\sL L$ is the free loop space of $L$. From this we deduce that when $L$ is a surface or a $K(\pi, 1)$, $\psi^1|_L$ is homotopic to the identity.\par 
        Using methods of \cite{Lalonde-McDuff}, we also show that given a family of Lagrangians all of which are Hamiltonian isotopic to $L$ over a sphere or a torus, the associated fibre bundle cohomologically splits over $\Z/2$. \par 
        %Along the way we find conditions for the moduli space of holomorphic discs with boundary on $L$ to be $R$-orientable for various $R$ which may be of independent interest. (Ask Ivan what I should say here- well-known to experts?)
    \end{abstract}
    \tableofcontents
    \section{Introduction}
        \subsection{Background}
            Let $(M^{2n}, \omega)$ be a symplectic manifold of dimension $2n$, and $L^n \subseteq M$ a Lagrangian submanifold. Let $\psi^t$ be a Hamiltonian isotopy of $M$ such that its time-1 flow preserves $L$ setwise, $\psi^1 (L) = L$. Its restriction to $L$ is then a self-diffeomorphism of $L$. We can consider the \emph{Hamiltonian monodromy group} $\sG_L \subseteq \operatorname{Diff}(L)$ of diffeomorphisms of $L$ arising in this way. A natural question is:
            \begin{QQ}\label{Question}
                What is $\sG_L$?
            \end{QQ}
            An elementary argument using the Weinstein neighbourhood theorem shows that if $f$ and $g$ are isotopic diffeomorphisms of $L$ and $g$ lies in $\sG_L$, then $f$ does too. This implies that $\sG_L$ is a union of connected components in $\operatorname{Diff}(L)$, and hence to study $\sG_L$, it suffices to study its image in the mapping class group $\pi_0 \operatorname{Diff}(L)$.\par 
            The subgroup $\sG_L$ was first studied by Yau in \cite{Mei-LinYau}, who proved the following theorem.
            \begin{TT}[\cite{Mei-LinYau}]
                Let $T$ and $T'$ be the standard monotone Clifford and Chekanov tori in $\C^2$ with the same monotonicity constant. Then in both cases $L = T$ or $T'$, $\pi_0 \sG_L \cong \Z / 2$.\par 
                However, there is no isomorphism $H_1(T) \cong H_1(T')$ which respects the $\Z / 2$ action on both and also preserves the Maslov index homomorphism.
            \end{TT}
            This provides a new proof that these two Lagrangians are not Hamiltonian isotopic.\par 
            Question \ref{Question} has been studied in few other places: by Hu, Lalonde and Leclercq in \cite{Hu-Lalonde-Leclercq}, by Mangolte and Welschinger in \cite{Mangolte-Welschinger}, by Evans and Rizell in \cite{ER14}, by Ono in \cite{Ono}, by Varolgunes in \cite{Varolgunes}, by Keating in \cite{Ailsa}, and by Augustynowicz, J. Smith and Wornbard in \cite{Jack}. Out of these, only Hu-Lalonde-Leclercq and Evans-Rizell focus on the case when $L$ is exact and embedded, which is where we will focus.\par
            We now assume (for the rest of the paper) the following:
            \begin{BB}
                \begin{enumerate}
                    \item $M$ is a product of symplectic manifolds which are either compact or Liouville. \par 
                    \item $L$ is compact, connected and relatively exact, i.e. $\omega$ vanishes on the relative homotopy group $\pi_2(M, L)$.\par 
                    \item $\psi^t$ is compactly supported for all $t$ and constant in $t$ near 0 and 1.
                \end{enumerate}
            \end{BB}
            The first two assumptions constrain the behaviour of the holomorphic curves that we will consider. Note that given $\psi^t$ not satisfying the third, we can always deform it (by multiplying the Hamiltonian by an appropriate cut-off function) so that it does, without changing $\psi^1|_L$, so this assumption does not lose any generality.\par 
            
            $\sG_L$ was studied by Hu, Lalonde and Leclercq in this more restrictive setting in \cite{Hu-Lalonde-Leclercq}, where they proved the following:
            \begin{TT}[\cite{Hu-Lalonde-Leclercq}]\label{HLL theorem}
                \begin{enumerate}
                    \item $\psi^1|_L$ acts as the identity on $H_*(L; \Z/2)$.
                    \item If $L$ admits a relative spin structure in $M$ which $\psi^1$ preserves, $\psi^1|_L$ acts as the identity on $H_*(L; \Z)$.
                \end{enumerate}
            \end{TT}
            It follows from Yau's results that relative exactness is a necessary condition here.\par 
            We will recover Theorem \ref{HLL theorem} as a special case of Theorem \ref{Main Theorem}. Note that in the case that $L$ is a homotopy sphere, this theorem does not give any information on the homotopy class of $\psi^1|_L$: if $\psi^1|_L$ were not homotopic to the identity, it would not preserve any orientation of $L$ and hence would not preserve any relative spin structure.\par 
            Our goals will be to weaken the orientation assumptions, extend this result to some other generalised cohomology theories, and extend this result to the free loop space of $L$. From this, we will fully compute $\sG_L$ in the case that $n = 2$. Our technical set-up will be very different to that of Hu, Lalonde and Leclercq, but our general approach is inspired by theirs.
            \begin{RR}
                Hu, Lalonde and Leclercq prove this using Morse and quantum cohomology (as in \cite{BiranCornea}), and one possible approach to extending Theorem \ref{HLL theorem} to other generalised cohomology theories would be to recreate Hu, Lalonde and Leclercq's proof using the methods of Cohen, Jones and Segal \cite{CJS, C}. However one may need to require stronger orientation hypotheses than Assumption \ref{Orientation Assumption}.
            \end{RR}
            \begin{RR}
                If $L$ not assumed to be relatively exact, Theorem \ref{HLL theorem} fails, as seen in Yau's example (\cite{Mei-LinYau}). However if $L$ is monotone, the failure for Theorem \ref{HLL theorem} to hold is understood- it can be seen in the failure of a certain element in $HF^*(L)$ to lie in the centre of $HF^*(L)$ (when $L$ is exact, $HF^*(L) \cong H^*(L)$ is always commutative)- see \cite[Remark 4.2]{Jack} and \cite{Varolgunes} for further discussions on this point.\par 
                It would be interesting to see if a similar phenomenon holds in the setting of other generalised cohomology theories, though there is not yet a construction of a Lagrangian Floer cohomology ring for non-relatively exact Lagrangians in the literature.
            \end{RR}
        \subsection{Extensions to generalised homology theories}
            We will define a space $\mathcal{D}_0$, which is roughly the space of smooth maps $D^2 \rightarrow M$ sending 1 to $L$, with moving Lagrangian boundary conditions on the rest of the boundary. This admits a map $\pi: \mathcal{D}_0 \rightarrow L$ by evaluating at 1. We will construct a virtual vector bundle called the \emph{index bundle}, $\Ind$, on any finite CW complex in $\mathcal{D}_0$ (compatible with restriction). After picking a (generic) almost complex structure $J$ on $M$ and some specific choices of moving Lagrangian boundary conditions, the tangent bundle of the moduli space of $J$-holomorphic discs with these boundary conditions will be stably isomorphic to the restriction of $\Ind$. In particular, $L$ embeds into $\mathcal{D}_0$ as the space of constant discs, with $\Ind$ restricting to $TL$.\par 
            We fix some ring spectrum $R$, and will recall definitions of ring spectra and $R$-orientability in Section \ref{Spectra Section}. We will make the following assumption, and in Section \ref{Orienty Subsection}, find conditions under which it holds.
            \begin{BB}\label{Orientation Assumption}
                The virtual vector bundle $\Ind - \pi^* TL$ is $R$-orientable.
            \end{BB}
            We will show:
            \begin{TT}\label{Main Theorem}
                Under Assumption \ref{Orientation Assumption}, the map 
                $$\psi^1|_L: \Sigma^\infty_+ L \wedge R \rightarrow \Sigma^\infty_+ L \wedge R$$
                is homotopic to the identity as a map of $R$-modules.
            \end{TT}
            \begin{RR}
                Our proof of this will use a minimal amount of technical machinery: we will only use standard Gromov compactness results and standard transversality results, and we will not need any form of gluing. It is possible to prove this without invoking any transversality results, and instead only use the fact that certain operators are Fredholm, using ideas in \cite{Hofer}; see \cite[Remark 2.13]{Me+Amanda}.
            \end{RR}
            From standard duality theory for spectra (Corollary \ref{dualid}), we deduce from Theorem \ref{Main Theorem}:
            \begin{CC}\label{COROLLARY}
                Under Assumption \ref{Orientation Assumption},
                \begin{enumerate}
                    \item $\psi^1|_L$ induces the identity map on $R_*(L)$.
                    \item $\psi^1|_L$ induces the identity map on $R^*(L)$.
                \end{enumerate}
            \end{CC}
            We will deduce Theorem \ref{Main Theorem} from the following sequence of lemmas.\par 
            In Section \ref{Moduli Section}, we will define a moduli space $\sP$ of holomorphic discs in $M$, with moving boundary conditions, lying naturally inside $\sD_0$. Its tangent space $T\sP$ will be stably isomorphic to $\Ind$. We will prove that it satisfies the following lemmas:
            \begin{LL}
                $\sP$ is a closed smooth manifold of dimension $n$.
            \end{LL}
            and
            \begin{LL}\label{Moduli Space Works}
                The following diagram commutes up to homotopy: 
                \[\begin{tikzcd}
                    & L \arrow[dd, "\psi^1"]\\
                    \sP \arrow[ur, "\pi"] \arrow[dr, "\pi"] \\
                    & L \\
                \end{tikzcd}\]
            \end{LL}
            Furthermore, by varying the boundary conditions in a 1-parameter family, we will construct a cobordism $\pi: \sW \to L$ from $Id: L \rightarrow L$ to $\pi: \mathcal{P} \rightarrow L$ over $L$, such that $T\sW - \pi^*TL$ will be $R$-orientable under Assumption \ref{Orientation Assumption}. A standard application of the Pontryagin-Thom construction (see Section \ref{Spectra Section}) will prove the following lemma:
            \begin{LL}\label{aqws}
                The induced map $\pi: \Sigma^\infty_+ \mathcal{P} \wedge R \rightarrow \Sigma^\infty_+ L \wedge R$ admits a section of $R$-module maps (up to homotopy of $R$-module maps).
            \end{LL}
            \begin{proof}
                This follows from the construction of $\sW$ in Section 4, along with Lemma \ref{Cobordism Implies Section}.
            \end{proof}
            \begin{proof}[Proof of Theorem \ref{Main Theorem}]
                Let $s$ be the section from Lemma \ref{aqws}. Then from Lemma \ref{Moduli Space Works}, we have that 
                $$\psi^1 \simeq \psi^1 \circ \pi \circ s \simeq \pi \circ s \simeq Id$$
                as maps of $R$-modules $\Sigma^\infty_+L \wedge R \rightarrow \Sigma^\infty_+L \wedge R$.
            \end{proof}
            In Section \ref{Orienty Subsection}, we will show the following:
            \begin{PP}\label{Orientation Criteria}
                Assumption \ref{Orientation Assumption} holds when:
                \begin{enumerate}
                    \item $R = H\Z / 2$ representing mod-2 singular homology.
                    \item $R = H\Z$ representing integral singular homology, and $w_2(L)(x) = 0$ whenever $x$ is a homology class represented by a 2-torus $S^1 \times \partial D^2$ in $L$ which can be extended to a solid torus $S^1 \times D^2$ in $M$.
                    %$$\pi_3(M, L) \rightarrow \pi_2(L) \rightarrow H_2(L; \Z)$$
                    \item $R = KU$ representing complex $K$-theory, and $L$ admits a spin structure.
                    \item $R = KO$ representing real $K$-theory, and $TL$ admits a stable trivialisation over a 3-skeleton of $L$ which extends (after applying $\cdot \otimes \C$) to a stable trivialisation of $TM$ over a 4-skeleton of $M$.
                    \item $R$ is a complex-orientable generalised cohomology theory and $L$ is stably parallelisable.
                    \item $R = \Sphere$ is the sphere spectrum and there is a real vector bundle $E$ over $X$, along with stable isomorphisms 
                    $TM \cong E \otimes \C$ and $TL \cong E|_L$ compatible with each other.
                    % and there is a stable trivialisation of $TL$ which (after applying $\cdot \otimes \C$) extends to a stable trivialisation of $TM$.
                \end{enumerate}
            \end{PP}
            \begin{RR}
                Our methods show that when $R = H\Z$, Assumption \ref{Orientation Assumption} holds if $L$ is spin. We will deduce the stronger statement (2) from a theorem of Georgieva (\cite[Theorem 1.1]{Georgieva}).
            \end{RR}
            \begin{CC}\label{sphere}
                When $L$ is a homotopy sphere, $\psi^1|_L$ is homotopic to the identity.
            \end{CC}
            \begin{proof}
                $L$ admits a spin structure, so $\psi^1$ acts as the identity on integral homology and so is homotopic to the identity.
            \end{proof}
            It is nontrivial to find examples of self-diffeomorphisms of spin manifolds which act trivially on integral cohomology but non-trivially on complex $K$-theory, so we provide an example in Appendix \ref{K-theory Example}, following a suggestion of Randal-Williams (\cite{R-W}). We furthermore provide an example in Appendix \ref{real K-theory Example} of a self-diffeomorphism of a manifold which acts trivially on integral cohomology but non-trivially on real $K$-theory.
        \subsection{Extensions to the free loop space}
            Fix a ring spectrum $R$. We will prove an analogue of Theorem \ref{Main Theorem} for the free loop space $\sL L$ of $L$, the space of maps $S^1 \rightarrow L$. Once again, we will need to make an assumption in order to orient our moduli spaces.
            \begin{TT}\label{Looped Theorem}
                Assume Assumption \ref{Looped Orientation Assumption} holds. Then the induced map
                $$\psi^1|_L: \Sigma^\infty_+ \sL L \wedge R\rightarrow \Sigma^\infty_+ \sL L \wedge R$$ 
                is homotopic to the identity, as maps of $R$-modules.
            \end{TT}
            \begin{proof}[Idea of proof]
                To prove this, we will define $\sL_1$ to be the space of free loops in the mapping torus $L_{\psi^1}$ of $\psi^1|_L$ which have winding number one over $S^1$, and whose basepoint lies in the fibre of the mapping torus over the basepoint in $S^1$. $\psi$ determines an automorphism $\Psi^1: \sL_1 \rightarrow \sL_1$, which one should think of as parallel transporting once around this mapping torus. \par 
                We will construct a map
                $$p \circ \left(\left[\sN\right] \cdot \right): \bigvee\limits_{j \in \Z} \Sigma^{\infty + j}_+ \sL_1 \wedge R \rightarrow \Sigma_+^\infty \sL L \wedge R$$
                This will use a moduli space of holomorphic discs with moving boundary conditions $\sN$, as well as Cohen and Jones' version of the Chas-Sullivan product, constructed similarly to \cite{Cohen}. We will then show:
                \begin{enumerate}
                    \item $\Psi^1$ acts as the identity on $\sL_1$ up to homotopy.
                    \item $p \circ \left(\left[\sN\right] \cdot \right)$ intertwines the actions of $\Psi^1$ and $\psi^1|_L$.
                    \item $p \circ \left(\left[\sN\right] \cdot \right)$ admits a section up to homotopy.
                \end{enumerate}
                Together, these will imply Theorem \ref{Looped Theorem}.
            \end{proof}
            \begin{CC}
                Assume Assumption \ref{Looped Orientation Assumption} holds. Then the map 
                $$\left(\psi^1|_L\right)_*: R_*(\sL L) \rightarrow R_*(\sL L)$$ 
                is the identity.
            \end{CC}
            \begin{PP}\label{MN Orientations}
                Assumption \ref{Looped Orientation Assumption} holds when
                \begin{enumerate}
                    \item $R = H\Z / 2$.
                    \item $R = H\Z$ and $L$ admits a spin structure which is preserved by $\psi^1$.
                    \item $R = KU$ and $L$ admits a spin structure which is preserved by $\psi^1$. %%%
                    % \item $R = KU$ and $L$ admits a homotopy class of stable trivialisations over a 2-skeleton which is preserved under $\psi^1|_L$.
                    \item $R$ is any complex-orientable generalised cohomology theory and $TL$ admits a homotopy class of stable trivialisations which is preserved under $\psi^1|_L$.
                \end{enumerate}
            \end{PP}
            From this, we can obtain strong restrictions on $\sG_L$ for certain $L$ that we could not get from Theorem \ref{Main Theorem}:
            \begin{CC}
                \begin{enumerate}
                    \item If $L$ is a $K(\pi_1(L), 1)$, $\psi^1|_L$ is homotopic to the identity.
                    \item If $n = 2$, $\psi^1|_L$ is isotopic to the identity, and hence $\sG_L$ is trivial.
                \end{enumerate}
            \end{CC}
            \begin{proof}
                Recall that a diffeomorphism of a closed surface is homotopic to the identity iff it is isotopic to the identity. Therefore to prove (2), it suffices to prove (1) along with the special cases $L = S^2$ and $\R\bP^2$. \par\par 
                It follows from Corollary \ref{sphere} that if $L = S^2$, then $\psi^1|_L$ is isotopic to the identity. Furthermore, since the mapping class group of $\R\bP^2$ is trivial, we can assume that $L$ is a $K(\pi_1(L), 1)$.\par 
                Note that $H_0(\sL L; \Z/2)$ is the free $\Z/2$-module generated by the set of homotopy classes of free loops in $L$. Therefore by Theorem \ref{Looped Theorem}, $\psi^1|_L$ acts as the identity on this set of generators. Then by \cite[Theorem 2.4]{Costoya-Viruel}, $\psi^1|_L$ must be homotopic to the identity.
            \end{proof}
        \subsection{Extensions to other bases}
            In this subsection and Section \ref{Other Bases Section}, we restrict to the case $R = H\Z / 2$.
            \begin{DD}
                Given a fibre bundle $E \twoheadrightarrow B$, we say it \emph{c-splits} if the inclusion of a fibre into $E$ induces an injection on mod-2 singular homology, for any fibre of this bundle.
            \end{DD}
            If the fibre $F$ has $H^*(F; \Z/2)$ finite-dimensional (as will be the case below), the Leray-Hirsch theorem says that this implies that
            $$H^*(E; \Z/2) \cong H^*(B; \Z/2) \otimes H^*(F; \Z/2)$$
            We define $Lag_L$ to be the space of (unparametrised) Lagrangian submanifolds of $M$ which are Hamiltonian isotopic to $L$. There is a natural fibre bundle $\sE$ over $Lag_L$, where the fibre over $K$ in $Lag_L$ is $K$. This is naturally a subbundle of $Lag_L \times M$. Given a map $\gamma: S^1 \rightarrow Lag_L$, Theorem \ref{HLL theorem} implies that the pullback bundle $\gamma^* \sE$ c-splits. We will use techniques in \cite{Lalonde-McDuff} to deduce the following from Theorem \ref{HLL theorem}:
            \begin{TT}\label{Bigger Bases}
                Let $X$ be a finite CW complex that admits a map $f: (S^1)^i \rightarrow X$ which is surjective on mod-2 homology (such as a product of spheres), and let $\gamma: X \rightarrow Lag_L$ be any map. Then the pullback bundle $\gamma^* \sE$ c-splits.
                % \begin{enumerate}[(i).]
                %     \item Given a map $\gamma: (S^1)^i \rightarrow Lag_L$, the pullback bundle $\gamma^* \sE \twoheadrightarrow (S^1)^i$ c-splits.
                %     \item Given a map $\gamma: S^i \rightarrow Lag_L$, the pullback bundle $\gamma^* \sE \twoheadrightarrow S^i$ c-splits.
                % \end{enumerate}
            \end{TT}
                
        \subsection{Acknowledgements}
            The author would like to thank Amanda Hirschi, Jack Smith, Nick Nho and Oscar Randal-Williams for helpful conversations regarding this project,  Ailsa Keating, Mohammed Abouzaid, Thomas Kragh, Jonny Evans, Oscar Randal-Williams and the two anonymous referees for useful comments and feedback on earlier drafts. The author would also like to express his deep gratitude to his supervisor Ivan Smith for all his support, motivation, productive discussions and ability to endure endless questions over the past few years. The author is supported by the Engineering and Physical Sciences Research Council [EP/W015889/1], and was also supported by an EPSRC PhD studentship [grant number 2261120].
    \section{Spectra and Pontryagin-Thom maps}\label{Spectra Section}
        We assume all spaces that we work with are Hausdorff, paracompact and homotopy equivalent to CW complexes. For an unbased space $X$, we will write $X_+$ for $X$ with a disjoint basepoint added.\par 
        We will use the category of spectra described in \cite{Rudyak}, whose construction and properties we will sketch here.
        \subsection{Spectra}
            \begin{DD}
                A \emph{spectrum} $X$ is a sequence $\{X_n\}_{n \in \mathbb{N}}$ of based spaces along with \emph{structure maps} $\Sigma X_n \rightarrow X_{n+1}$.\par 
                A \emph{function} $f$ from $X$ to $Y$ is a family of based maps of spaces $f_n: X_n \rightarrow Y_n$ such that the square\\
                \[\begin{tikzcd}
                    \Sigma X_n \arrow[r, "\Sigma f_n"] \arrow[d] & \Sigma Y_n \arrow[d]\\
                    X_{n+1} \arrow[r, "f_{n+1}"] & Y_{n+1}
                \end{tikzcd}\]
                commutes.
            \end{DD}
            We will not define a morphism of spectra (see \cite[\S II.1]{Rudyak} for details), but any function of spectra is a morphism of spectra, and all morphisms of spectra that we will directly construct will arise in this way.
            \begin{RR}
                Rudyak actually defines a spectrum in a slightly different way but shows the two are equivalent in \cite[Lemma II.1.19]{Rudyak}.
            \end{RR}
            \begin{DD}
                Given a spectrum $X$ and a based space $Z$, there is a spectrum $X \wedge Z$ with $(X \wedge Z)_n = X_n \wedge Z$.\par 
                A \emph{homotopy} between two functions $f, g: X \rightarrow Y$ is a function $$H: X \wedge [0,1]_+ \rightarrow Y$$
                restricting to $f$ over $\{0\}_+$ and to $g$ over $\{1\}_+$. \par 
                One can extend this definition to homotopies of morphisms of spectra, as in \cite[Definition II.1.9]{Rudyak}, and for homotopic morphisms $f$ and $g$ we write $f \simeq g$. Then for spectra $X$ and $Y$, we define $[X, Y]$ to be the set of homotopy classes of morphisms $X \rightarrow Y$. There is a natural structure of an abelian group on this set.\par 
                From this definition of a homotopy we can define a notion of \emph{(homotopy) equivalence} as with spaces, which we denote by $\simeq$.
            \end{DD}
            There is a natural functor $\Sigma^\infty$ from based spaces to $Sp$ sending a space $X$ to the spectrum with $(\Sigma^\infty X)_n = \Sigma^n X$, with structure maps the identity. We write $\Sigma^\infty_+$ for the functor from unbased spaces to $Sp$ sending $X$ to $\Sigma^\infty X_+$. These both send homotopies to homotopies. We will denote $\Sigma^\infty_+ \{*\}$ by $\mathbb{S}$, and call this the \emph{sphere spectrum}. This has $\Sphere_n = S^n$, justifying the name.\par 
            For any $N$ in $\mathbb{Z}$, there is an endofunctor $\Sigma^N$ of $Sp$ which sends $X$ to the spectrum $\Sigma^N X$ with
            $$(\Sigma^N X)_n =
            \begin{cases}
                X_{n+N} & \text{ if }n + N \geq 0\\
                \{*\} & \text{ Otherwise}
            \end{cases}$$
            This satisfies $\Sigma^N \Sigma^M X \simeq \Sigma^{N+M}X$ for all $X$, and $[X, Y] = [\Sigma^N X, \Sigma^N Y]$ for all $X, Y$. We write $\Sigma^{\infty + N}_+ X$ for $\Sigma^N \Sigma^{\infty}_+ X$.
            \begin{DD}
                Given a family $\{X_\lambda\}_{\lambda \in \Lambda}$ of spectra, we can define their \emph{wedge product} to be the spectrum $\bigvee\limits_\lambda X_\lambda$ with
                $$\left(\bigvee\limits_\lambda X_\lambda \right)_n = \bigvee\limits_\lambda X_n$$
            \end{DD}
            \begin{LL}[{\cite[Proposition II.1.16]{Rudyak}}]
                For any spectrum $Y$, there are natural isomorphisms
                $$\left[\bigvee\limits_\lambda X_\lambda, Y \right] \cong \prod\limits_\lambda \left[ X_\lambda, Y \right]$$
                and
                $$\left[Y, \bigvee\limits_\lambda X_\lambda \right] \cong \bigoplus \limits_\lambda \left[Y, X_\lambda \right]$$
            \end{LL}
            We can extend the smash product $\wedge$ to two spectra, functorially in each argument, as in \cite{Adams}. We will not need the explicit construction but will use some of its properties:
            \begin{TT}[{\cite[Theorems II.2.1 and II.2.2]{Rudyak}}]
                There are equivalences
                \begin{enumerate}
                    \item $(X \wedge Y) \wedge Z \simeq X \wedge (Y \wedge Z)$
                    \item $X \wedge Y \simeq Y \wedge X$
                    \item $X \wedge \mathbb{S} \simeq X$
                    \item $\Sigma X \wedge Y \simeq \Sigma(X \wedge Y)$
                \end{enumerate}
                such that all natural diagrams made up of these equivalences commute up to homotopy.
            \end{TT}
            \begin{DD}
                A \emph{ring spectrum} $R$ is a spectrum equipped with a unit morphism $\eta: \mathbb{S} \rightarrow R$ and product morphism $\mu: R \wedge R \rightarrow R$, satisfying appropriate associativity and unitality conditions up to homotopy.\par 
                A \emph{(right) module} over a ring spectrum is defined similarly.\par 
                A \emph{map} of $R$-modules $X \rightarrow Y$ is a map of spectra $X \rightarrow Y$ which commutes with the actions of $R$ on $X$ and $Y$ up to homotopy.\par
                A \emph{homotopy} of $R$-module maps between $f$ and $g: X \rightarrow Y$ is a homotopy of spectra $X \wedge [0, 1]_+ \rightarrow Y$ between $f$ and $g$, which is also a map of $R$-modules.
            \end{DD}
            \begin{EE}
                The sphere spectrum $\Sphere$ is naturally a ring spectrum, and every spectrum is naturally a module spectrum over it.
            \end{EE}
            \begin{EE}
                For a ring spectrum $R$, there is a functor from spectra to (right) $R$-module spectra, given by $\cdot \wedge R$.
            \end{EE}
            \begin{DD}
                We define the \emph{stable homotopy groups} of $X$ to be $$\pi_i X := [\Sigma^i \mathbb{S}, X]$$
                This is covariantly functorial in $X$.\par 
                Given a spectrum $R$, we define $R_i(X)$ to be $\pi_i(X \wedge R)$ and $R^i(X)$ to be $[X, \Sigma^i R]$. These are functorial in $X$, covariantly and contravariantly respectively.\par 
                Given a space $Z$, we define $R_*(Z)$ to be $R_*(\Sigma^\infty_+ Z)$ and $R^*(Z)$ to be $R^*(\Sigma^\infty_+ Z)$. These are functorial in $Z$, covariantly and contravariantly respectively. 
            \end{DD}
            By Brown's representability theorem (\cite[Theorem 4E.1]{Hatcher}), for an abelian group $G$ there is a (unique up to homotopy equivalence) spectrum $HG$ such that $HG_*$ and $HG^*$ are homology and cohomology with co-efficients in $G$ respectively, and when $G$ is a ring these are in fact ring spectra. Similarly there are (unique up to homotopy equivalence) ring spectra $KO$ and $KU$ such that $KO^*$ and $KU^*$ are real and complex $K$-theory respectively.
        \subsection{Duality}
            We say a spectrum $X$ is \emph{finite} if it is of the form $X = \Sigma^{\infty + i} Y$ for $Y$ a based finite CW complex, and $i$ in $\Z$. Finite spectra admit duals in the following sense.
            \begin{LL}[{\cite[Corollary II.2.9]{Rudyak}}]\label{duality}
                For a finite spectrum $X$, there is another finite spectrum $X^\vee$, unique up to natural homotopy equivalence, along with natural isomorphisms 
                $$[X, E] \cong [\Sphere, X^\vee \wedge E]$$
                and 
                $$[X^\vee, E] \cong [\Sphere, X \wedge E]$$
                for any other spectrum $E$. These induce further natural isomorphisms
                $$E_*(X) \cong E^{-*}(X).$$
            \end{LL}
            We will also use a base-change isomorphism. Let $R$ be a ring spectrum, with unit map $i: \Sphere \rightarrow R$.
            \begin{LL} \label{basechange}
                Let $X$ be a spectrum and $M$ an $R$-module, with $R$-action $\mu: M \wedge R \rightarrow M$. Then there is a natural isomorphism
                $$[X, M] \cong [X \wedge R, M]^R$$
                where $[\cdot, \cdot]^R$ denotes homotopy classes of $R$-module maps.
            \end{LL}
            \begin{proof} 
                Let $\alpha: [X, M] \rightarrow [X \wedge R, M]^R$ send $\phi$ to the composition
                $$X \wedge R \xrightarrow{\phi \wedge R} M \wedge R \xrightarrow{\mu} M$$
                and let $\beta: [X \wedge R, M]^R \rightarrow [X, M]$ send $\psi$ to the composition
                $$X \xrightarrow{X \wedge i} X \wedge R \xrightarrow{\psi} M.$$
                Then $\alpha$ and $\beta$ are inverses to each other.
            \end{proof}
            From these two lemmas we deduce the following corollary.
            \begin{CC}\label{moddual}
                For finite spectra $X$ and $Y$, there is an isomorphism
                $$[X \wedge R, Y \wedge R]^R \cong [Y^\vee \wedge R, X^\vee \wedge R]^R.$$
            \end{CC}
            \begin{proof}
                We have isomorphisms
                \begin{align*}
                    [X \wedge R, Y \wedge R]^R & \cong [X, Y \wedge R]\\
                    & \cong [Y^\vee \wedge X, R]\\
                    & \cong [Y^\vee, X^\vee \wedge R]\\
                    & \cong [Y^\vee \wedge R, X^\vee \wedge R]^R.
                \end{align*}
                The first and fourth isomorphisms are from Lemma \ref{basechange}, and the second and third isomorphisms are from Lemma \ref{duality}.
            \end{proof}
            \begin{CC}\label{dualid}
                Let $X$ be a finite spectrum and $f: X \rightarrow X$ be a map such that $f\wedge R: X \wedge R \rightarrow X \wedge R$ is homotopic to the identity as a map of $R$-modules. Then $f$ induces the identity map on both $R_*(X)$ and $R^*(X)$.
            \end{CC}
            \begin{proof}
                $R_*(X) = \pi_*(X \wedge R)$, so $f$ induces the identity map on $R_*(X)$.\par 
                When $X = Y$, the isomorphism of Corollary \ref{moddual} sends the identity to the identity, so the induced map of $R$-modules $X^\vee \wedge R \rightarrow X^\vee \wedge R$ is homotopic to the identity. But the action of this map on homotopy groups is the same as the action of $f$ on $R^*(X)$.
            \end{proof}
        \subsection{Thom spectra}
            Let $\xi: E \rightarrow X$ be a vector bundle of rank $n$ over a space $X$. We let $DE$ and $SE$ denote the unit disc and unit sphere bundles of $E$ respectively, with respect to some choice of metric.
            \begin{DD}
                We define the \emph{Thom space} of $E$, denoted $X^E_u$, to be the quotient $DE / SE$. This is a based space with basepoint given by the image of $SE$.\par 
                We define the \emph{Thom spectrum} of $E$, denoted $X^E$, to be the spectrum $\Sigma^\infty X^E_u$.\par 
                If $E'$ is a virtual vector bundle which can be written as $E' \cong E - \mathbb{R}^m_X$ for some $m$ (in particular, this includes any virtual vector bundle pulled back from one over a compact space), we define the \emph{Thom spectrum} of $E'$ to be $\Sigma^{\infty-m} X_u^E$.
            \end{DD}
            \begin{RR}\label{Thom Spectrum Is Stable}
                None of these depend on the choice of metric up to homotopy equivalence, and furthermore the Thom spectrum only depends on the stable isomorphism class of the virtual vector bundle by \cite[Lemma IV.5.14]{Rudyak}.
            \end{RR}
            Now let $\xi: E \rightarrow X$ be a virtual vector bundle of rank $n$.
            \begin{DD}
                $E$ is \emph{orientable} with respect to $R$ if there is a morphism $U: X^E \rightarrow \Sigma^{n}R$, called the \emph{Thom class}, whose restriction to a fibre (which is equivalent to a copy of $\Sigma^n \Sphere$) represents plus or minus the unit in $R$.\par 
                An \emph{orientation} is a homotopy class of such morphisms.
            \end{DD}
            It follows from Remark \ref{Thom Spectrum Is Stable} that being $R$-orientable only depends on the stable isomorphism class of $E$.
            \begin{LL} [{\cite[Proposition V.1.10 and Examples V.1.23]{Rudyak}}]\label{Orientabiliy Conditions from Rudyak}
                \begin{enumerate}
                    \item Any virtual vector bundle is canonically oriented with respect to $H\Z/2$.
                    \item A virtual vector bundle is orientable with respect to $H\Z$ iff it is orientable in the usual sense, and there is a natural bijection between $H\Z$-orientations and orientations in the usual sense.
                    \item A trivialisation of a virtual vector bundle induces a natural orientation with respect to any $R$.
                    \item If $E$ is oriented with respect to $R$ and $f: Y \rightarrow X$ is any map, the pullback bundle $f^*E$ admits a natural $R$-orientation.
                    \item If $F$ is another virtual vector bundle over $X$ and any two of $E$, $F$ and $E \oplus F$ are $R$-oriented, then the third admits a natural $R$-orientation.
                \end{enumerate}
            \end{LL}
            We will need a stable version of the Thom isomorphism theorem:
            \begin{TT}[{\cite[Theorem V.1.15 and Exercise V.1.28]{Rudyak}}]
                An $R$-orientation of $E$ induces an equivalence of $R$-module spectra
                $$\Sigma^{\infty + n}_+ X \wedge R \simeq X^E \wedge R$$
                More generally, if $F \rightarrow X$ is another virtual vector bundle, then an $R$-orientation of $E$ induces an equivalence of $R$-module spectra
                $$\Sigma^{n} X^F \wedge R \simeq X^{E \oplus F} \wedge R$$
            \end{TT}
        \subsection{Pontryagin-Thom collapse maps}
            Let $i: X \hookrightarrow Y$ be a smooth embedding of manifolds with $X$ compact, such that $i(\partial X) \subseteq \partial Y$. Let $\nu$ be the normal bundle of $i$, with unit disc bundle $D\nu$ and unit sphere bundle $S\nu$ (with respect to some choice of metric). Furthermore we fix a tubular neighbourhood of $X$, i.e. we pick an embedding $j: D\nu \hookrightarrow Y$ extending $i$, and not touching $\partial Y$ apart from over $\partial X$.\par
            We will denote the one-point compactification of $Y$ by $Y_\infty$, viewed as a based space with basepoint at infinity. If $Y$ is already compact then this is $Y_+$.
            \begin{DD}
                We define the \emph{Pontryagin-Thom collapse map} $i_{!,u}: Y_\infty \rightarrow X^\nu_u$ by
                $$ i_{!,u}(x) = 
                \begin{cases}
                    j^{-1}(x) & \textit{ if } x \in j(D\nu)\\
                    S\nu & \textit{ otherwise}
                \end{cases} $$
                We denote the stabilisation $\Sigma^\infty i_{!, u}$ by $i_!$.
            \end{DD}
            The map $i_{!,u}$ is continuous, and both the space $X^\nu_u$  and the map $i_{!,u}$ are independent of the choices of metric and tubular neighbourhood up to homotopy.\par
            Now let $f: X^m \rightarrow Y^n$ be any smooth map of manifolds. 
            \begin{DD}
                We define the \emph{stable normal bundle} of $f$, $\nu_f$, to be the virtual vector bundle $f^* TY - TX$.
            \end{DD}
            When $f$ is an embedding, this is stably isomorphic to the normal bundle of the embedding, defined in the usual sense.\par 
            Now assume $X$ and $Y$ are closed, and choose a smooth embedding 
            $$i: X \hookrightarrow \mathbb{R}^N$$
            for some $N$. Now $\nu_f \oplus \R^N$ is stably isomorphic to the normal bundle $\nu_{f \times i}$ of the embedding 
            $$f \times i: X \hookrightarrow Y \times \R^N$$ 
            We can consider the map
            $$(f \times i)_!: \Sigma^\infty Y_\infty \simeq \Sigma^{\infty-N} \left(Y \times \mathbb{R}^N\right)_\infty \rightarrow \Sigma^{-N} X^{\nu_{f \times i}} \simeq X^{\nu_f}$$
            \begin{LL}
                Up to homotopy, the morphism $\Sigma^{\infty} Y_\infty \rightarrow X^{\nu_f}$ does not depend on the choice of $i$ or $N$.
            \end{LL}
            \begin{proof}
                First, observe that if we compose $i$ with the standard embedding $\mathbb{R}^N \hookrightarrow \mathbb{R}^{N+1}$, the resulting map is the same up to suspension.\par 
                Assume $i$, $i'$ are two choices of embedding $X \hookrightarrow \mathbb{R}^N$. Pick an embedding 
                $$I: X \times [0,1] \hookrightarrow \mathbb{R}^N \times [0,1]$$
                over $[0,1]$ restricting to $i$ and $i'$ over $0$ and $1$ respectively, increasing $N$ if necessary as above. Then $(f \times I)_{!, u}$ provides a homotopy between $(f \times i)_{!,u}$ and $(f \times i')_{!, u}$.
            \end{proof}
            Now that we know the desuspended map $\Sigma^\infty Y_\infty \rightarrow X^{\nu_f}$ doesn't depend on these choices (up to homotopy), we denote it by $f_!$.\par 
            We fix a ring spectrum $R$. Given an $R$-orientation of $\nu_f$, we consider the following composition, which is a map of $R$-modules:
            $$p_f: \Sigma^\infty Y_\infty \wedge R \xrightarrow{f_!} X^{\nu_f} \wedge R \xrightarrow{\simeq} \Sigma^{\infty + d}_+ X \wedge R \xrightarrow{f} \Sigma^{\infty + d} Y_\infty \wedge R$$
            where $d = n - m$, and the middle map is from the Thom isomorphism theorem. By construction, this factors through $\Sigma^{\infty + d}_+ X \wedge R$.\par
            Now suppose that $n = m$ (so $d = 0$), $Y$ is compact (so $Y_\infty = Y_+$), and there is a cobordism $W$ from $X$ to $Y$ along with a map $F: W \rightarrow Y$ restricting to the identity on $Y$ and to $f$ on $X$, as in the following diagram:
            
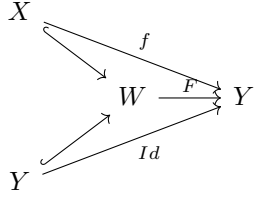
\[\begin{tikzcd}
                X & & \\
                & W & Y \\
                Y & & 
                \arrow[from=1-1, to=2-3, "f"]
                \arrow[from=1-1, to=2-2, hook]
                \arrow[from=3-1, to=2-3, "Id"']
                \arrow[from=2-2, to=2-3, "F"]
                \arrow[from=3-1, to=2-2, hook]
            \end{tikzcd}\]
            \begin{RR}
                We will find ourselves in this situation in the proof of Theorem \ref{Main Theorem}, with $Y = L$, $X = \sP$ a moduli space of suitable holomorphic discs with moving Lagrangian boundary conditions, and $\sW$ a moduli space of moving boundary conditions varying in a family of such boundary conditions. 
            \end{RR}
            We assume that $\nu_F$ admits an orientation with respect to $R$, which, by restriction, induces one for $\nu_f$.
            \begin{LL}\label{Cobordism Implies Section}
                After applying $\Sigma^\infty_+ \cdot \wedge R$, $f$ admits a section of $R$-module maps up to homotopy (of $R$-module maps).
            \end{LL}
            \begin{proof}
                We will show that the map $p_f$ constructed above is an equivalence, then $f_! \circ p_f^{-1}$ will be the desired section.\par 
                We pick an embedding 
                $$I: W \hookrightarrow \mathbb{R}^N \times [0,1]$$
                for some $N$, such that $I^{-1}\left(\R^N \times \{0\}\right) = X$ and $I^{-1}\left(\R^N \times \{1\}\right) = Y$.
                Then choosing an orientation of $\nu_F$ and performing the above construction to $F$ gives us a homotopy from $p_f$ to $p_{Id}$, where $p_{Id}$ is constructed similarly to $p_f$ but for $Id: Y \rightarrow Y$, using some choice of trivialisation on the trivial vector bundle over $Y$. But $p_{Id}$ is a composition of three equivalences and hence $p_f$ is an equivalence.\par 
                All the maps and homotopies are constructed either by applying $\cdot \wedge R$ to a map of spectra or by applying the Thom isomorphism theorem over $R$, so they are all $R$-module maps.
            \end{proof}
            \begin{CC}\label{Surjectivity Corollary}
                The map $f_*: R_* X \rightarrow R_* Y$ is split surjective.
            \end{CC}
            It follows from \cite[Lemma II.2.4 and Theorem V.2.3]{Rudyak} that 
            \begin{CC}\label{Injectivity Corollary}
                If $W$ is $R$-orientable (and hence $X$ and $Y$ are too), the map $f^*: R^*Y \rightarrow R^*X$ is injective.
            \end{CC}
            We can also define Pontryagin-Thom collapse maps without smoothness assumptions. Let $i: X \hookrightarrow Y$ be an embedding of spaces.
            \begin{DD}
                We say $X$ admits a \emph{tubular neighbourhood} in $Y$ with \emph{normal bundle} $\nu$ if there is a vector bundle $\nu$ over $X$ and an open neighbourhood $U$ of $X$ in $Y$, such that there is a homeomorphism $\phi: U \rightarrow DE$ between $U$ and the unit disc bundle $DE$ of $\nu$ (for some choice of fibrewise metric on $\nu$), sending $X$ to the zero section.
            \end{DD}
            \begin{DD}\label{PT without smoothness}
                Let $E \rightarrow Y$ be another vector bundle. We define the \emph{Pontryagin-Thom collapse map} $i_{!, u}: Y^E_u \rightarrow X^{\nu \oplus E}_u$ by
                $$(x, v) \mapsto \begin{cases}
                    (\phi(x), v) & \mathrm{ if } \,x \in U\\
                    S(\nu \oplus E) & \mathrm{ otherwise}
                \end{cases}$$
                We write $i_!$ for the induced map on Thom spectra.\par 
                This definition extends to the case when $E$ is a virtual vector bundle of the form $E' - \R^N$, as we can apply this construction to $E'$ and desuspend. Similarly to before, up to homotopy $i_! := \Sigma^\infty i_{!, u}$ then only depends on the virtual vector bundle up to stable isomorphism.
            \end{DD}
            \begin{EE}
                When $X \hookrightarrow Y$ is a proper embedding of smooth manifolds, $X$ always admits a tubular neighbourhood in $Y$, and we recover the earlier construction.
            \end{EE}
            
        \subsection{Fundamental classes}\label{Fundamental Classes Subsection}
            Let $X^n$ be a closed manifold of dimension $n$. Choose an embedding 
            $$i: X \hookrightarrow \R^N \hookrightarrow S^N$$
            for some $N$, with normal bundle $\nu$. Then as virtual vector bundles, there is a natural isomorphism $\nu - \R^N \cong -TX$. Consider the map $i_{!, u}: S^N \rightarrow X^\nu$. Stabilising and desuspending gives us a well-defined (up to homotopy) map of spectra
            $$[X]: \Sphere \rightarrow X^{-TX}$$
            Given a ring spectrum $R$ and an $R$-orientation of $TX$ (which induces one on $-TX$), the Thom isomorphism theorem gives us a well-defined (up to homotopy) map of spectra
            $$[X]: \Sphere \rightarrow \Sigma^{\infty-n}_+ X \wedge R$$
            representing an element $[X]$ in $R_n(X)$.
            More generally, given a space $Z$, a map of spaces $f: X \rightarrow Z$, a vector bundle $E \rightarrow Z$, and an $R$-orientation of $TX - f^*E$ (which we assume to have rank $d$), we can use the Thom isomorphism theorem to obtain a well-defined (up to homotopy) map of spectra
            $$\Sphere \xrightarrow{[X]} X^{-TX} \wedge R \simeq \Sigma^{-d} X^{-f^*E} \wedge R \rightarrow \Sigma^{-d}Z^{-E} \wedge R$$
            which we also denote by $[X]$. Here the middle homotopy equivalence is the Thom isomorphism.\par            
            We call all of these maps $[X]$ \emph{fundamental classes}.
            \begin{LL} \label{PT functorial}
                Let $X$ and $Y$ be closed manifolds, and let $f: X \hookrightarrow Y$ be an embedding. Then the following diagram commutes:
                \[ \begin{tikzcd}
                    \Sphere 
                    \arrow[d, "{[Y]}"] \arrow[dr, "{[X]}"] & \\
                    Y^{-TY} \arrow[r, "f_\Shrek"] &
                    X^{-TX}
                \end{tikzcd}\]
            \end{LL}
            \begin{proof}
                Let $i: Y \hookrightarrow \R^N$ be an embedding. Then $i \circ f: X \hookrightarrow \R^N$ is an embedding too, and using this embedding to construct the unstable Pontryagin-Thom collapse map, we get a commutative diagram (of spaces):
                \[ \begin{tikzcd}
                    S^N
                    \arrow[d, "i_{\Shrek, u}"] \arrow[dr, "(i \circ f)_{\Shrek, u}"] & \\
                    Y^{\nu_i} \arrow[r, "f_{\Shrek, u}"] &
                    X^{\nu_{i \circ f}}
                \end{tikzcd}\]
                where $\nu_i$ and $\nu_{i \circ f}$ are the normal bundles of $i$ and $i \circ f$ respectively.\par 
                Stabilising this diagram gives the required diagram.
            \end{proof}
            \begin{LL}\label{Cobodism Implies Same Fundamental Class}
                Let $Y$ be another closed manifold of dimension $n$. Suppose $f: X \rightarrow Z$ and $g: Y \rightarrow Z$ are two maps, such that there is a cobordism $W$ from $X$ to $Y$, and a map $F: W \rightarrow Z$ extending $f$ and $g$, as shown below.
                
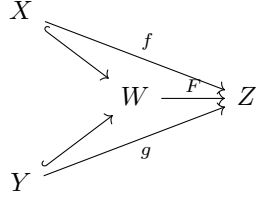
\[\begin{tikzcd}
                    X & & \\
                    & W & Z \\
                    Y & & 
                    \arrow[from=1-1, to=2-3, "f"]
                    \arrow[from=1-1, to=2-2, hook]
                    \arrow[from=3-1, to=2-3, "g"']
                    \arrow[from=2-2, to=2-3, "F"]
                    \arrow[from=3-1, to=2-2, hook]
                \end{tikzcd}\]
                Let $E \rightarrow Z$ be a vector bundle and assume there is an $R$-orientation of $TW - F^*E$, which we assume to have rank $d + 1$. Then the two fundamental classes $[X], [Y]$ in $R_d(Z^{-E})$ (defined using the orientations of $TX - f^*E$ and $TY - g^*E$ given by restricting the orientation of $TW - F^* E$) agree. 
            \end{LL}
            \begin{proof}
                Pick an embedding 
                $$I: W \hookrightarrow \R^N \times [0, 1]$$
                for some $N$, such that $W^{-1} \left(\R^N \times \{0\}\right) = X$ and $W^{-1} \left(\R^N \times \{1\}\right) = Y$. Then the composition
                $$\Sphere \wedge [0, 1]_+ \xrightarrow{I_!} W^{\R - TW} \wedge R \simeq \Sigma^{-d} W^{-F^* E} \rightarrow \Sigma^{-d} Y^{-E}$$
                where the middle map is the Thom isomorphism, defines a homotopy between $[X]$ and $[Y]$.
            \end{proof}
            \begin{RR}
                Let $X = \bigsqcup\limits_{i \geq 0} X^i$ be a manifold consisting of components $X^i$ of dimension $i$. If $X$ is compact (in particular implying that only finitely many components $X^i$ are non-empty), we can consider the fundamental classes of these independently and obtain a fundamental class $[X] := \sum\limits_i [X^i]$ in $\bigoplus\limits_{i} R_i\left(X^{-TX}\right)$. We will use this later in Section \ref{Looped Section}.
            \end{RR}
       
    \section{Index bundles and their orientations}\label{Index Bundle Section}
        \subsection{Cauchy-Riemann operators}\label{CR Operators Subsection}
            Let $B$ be a space and $A \subseteq B$ some subspace.
            \begin{DD}
                A \emph{bundle pair} over the pair $(B, A)$ is a complex vector bundle $E$ over $B$ along with a totally real subbundle $F$ of $E|_A$. This is written as
                $$(E, F) \rightarrow (B, A)$$ 
                We call $E$ the \emph{complex part} and $F$ the \emph{real part}.\par 
                We can perform certain operations on bundle pairs, analogously to vector bundles.
                \begin{itemize}
                    \item Given a real vector bundle $G$ over $B$, we define $(E, F) \otimes G$
                    to be the bundle pair
                    $$(E \otimes_\R G, F \otimes_\R G|_A)$$
                    \item Given $(E', F') \rightarrow (B, A)$ another bundle pair, we define ${(E, F) \oplus (E', F')}$
                    to be the bundle pair 
                    $$(E \oplus E', F \oplus F')$$
                    \item A \emph{virtual bundle pair} is a formal difference $(E, F) - (E', F')$ of two bundle pairs, and we say two virtual bundle pairs $(E, F) - (E', F')$ and ${(G, H) - (G', H')}$ are \emph{stably isomorphic} if there is an isomorphism of bundle pairs
                    $$(E, F) \oplus (G', H') \oplus (\C^N_B, \R^N_A) \cong (E', F') \oplus (G, H) \oplus (\C^N_B, \R^N_A)$$
                    for some $N$. \par 
                    \item The \emph{rank} of a bundle pair $(E, F)$ is $\Rank_\C E$, and the \emph{(virtual) rank} of a virtual bundle pair $(E, F) - (E', F')$ is $\Rank_\C E - \Rank_\C E'$.\par 
                    \item Given a bundle pair $(E, F) \rightarrow (B, A)$ and a map of pairs $f: (B', A') \rightarrow (B, A)$, we define the bundle pair $f^*(E, F) \rightarrow (B', A')$ to be the bundle pair $(f^*E, f^*F)$.\par 
                    \item We define a \emph{section} of $(E, F)$ to be a smooth section of $E$ whose restriction to $A$ lies in $F$. We write $\Gamma(E, F)$ for the space of sections.
                \end{itemize}
            \end{DD}
            Similarly to the case of real or complex vector bundles, when $B$ is a finite CW complex and $A$ a subcomplex, every virtual bundle pair over $(B, A)$ is stably isomorphic to $(E, F) - (\C^N_B, \R^N_A)$ for some $(E, F)$ and some $N$. Similarly the pullbacks of a bundle pair under homotopic maps are isomorphic.\par 
            We will usually take $(B, A)$ to be $(D^2, \partial D^2) \times X$ for some space $X$, and we will assume this to be the case for the rest of Section \ref{Index Bundle Section}.\par 
            For the rest of Section \ref{CR Operators Subsection}, we will assume that $X$ is a point, so $(E, F)$ is a bundle pair over $(D^2, \partial D^2)$.
            \begin{DD}
                A {\em (real) Cauchy-Riemann operator} on $(E, F)$ is an $\mathbb{R}$-linear first order differential operator
                $$D: \Gamma(E, F) \rightarrow \Omega^{0,1}(E)$$
                satisfying the Leibniz rule:
                $$D(f \eta) = (\bar{\partial} f) \eta + f D \eta$$
                for $f$ in $C^\infty(D^2, \R)$ and $\eta$ in $\Gamma(E, F)$.\par
            \end{DD}
            \begin{LL}[{\cite[Remark C.1.2]{McDuff-Salamon}}]\label{CR operators contractible}                 The space of Cauchy-Riemann operators on $(E, F)$ is contractible (and in fact, convex).
            \end{LL}
            For a choice of Hermitian metric on $E$ and a real number $q > 2$ (which we fix), a Cauchy-Riemann operator induces an operator 
            $$\hat{D}: W^{1, q}(E, F) \rightarrow W^{q}\left(\Lambda^{0, 1}T^*D^2 \otimes E\right)$$
            where these spaces are appropriate Sobolev completions of the above spaces of smooth sections. By \cite[Theorem C.1.10]{McDuff-Salamon}, this operator $\hat{D}$ is Fredholm, and in fact $\Ker \hat{D} = \Ker D$.
        \subsection{Index bundles}\label{IndexBundleSection}
            Let $X$ be a finite CW complex, and let 
            $${(E, F) \rightarrow (D^2, \partial D^2) \times X}$$
            be a bundle pair.\par 
            By Lemma \ref{CR operators contractible}, we can choose a Cauchy-Riemann operator $D_x$ on 
            $$\left(E|_{D^2 \times \{x\}}, F|_{\partial D^2 \times \{x\}}\right)$$
            for each $x$, varying continuously in $x$. The space of such choices is contractible.
            \begin{DD}
                Assume $\hat{D}_x$ is surjective for all $x$. The {\em index bundle} of $(E, F)$ is the (real) vector bundle $\Ind(E, F)$ over $X$, with fibre at a point $x$ given by $\Ker \hat{D}_x$. 
            \end{DD}
            If $\hat{D}_x$ is not always surjective, we can still define the index bundle, following \cite{Atiyah}. Since $X$ is compact, we can find, for some finite $N$, a continuous family of linear maps 
            $$\phi_x: \mathbb{R}^N \rightarrow W^q\left(\Lambda^{0, 1}T^*D^2 \otimes E|_{D^2 \times \{x\}}\right)$$
            for $x$ in $X$, such that the \emph{stabilised operator} $T_x := \hat{D}_x + \phi_x $
            $$T_x: W^{1, q}\left(E|_{D^2 \times \{x\}}, F|_{\partial D^2 \times \{x\}}\right) \oplus \R^N \rightarrow W^q\left(\Lambda^{0, 1} T^*D^2 \otimes E|_{D^2 \times \{x\}}\right)$$
            is surjective. In this situation we call $\phi$ a \emph{stabilisation} of \emph{rank} $N$. We then define $\Ind(E, F) + \R^N_X$ to be the vector bundle with fibre over $x$ given by $\Ker T_x$. 
            %This gives us a well-defined virtual vector bundle $\Ind(E, F)$, independent of all choices.\par
            \begin{LL}\label{Index Bundles Are Canonical}
                The vector bundle $\Ind(E, F)$ is well-defined up to stable isomorphism. Furthermore this stable isomorphism can be chosen in a way that is unique up to weakly contractible choice.
            \end{LL}
            \begin{proof}
                Firstly we note that stabilising $\phi$ by adding another copy of $\R$ which is sent to 0 does not change the index bundle up to canonical isomorphism.\par 
                We choose a metric on $E$, and note that the space of such choices is contractible. This defines a $W^{0, 2}$ inner product on $W^{1, q}(E, F)$.\par 
                Given two choices of family of Cauchy-Riemann operators and stabilisation $(D, \phi)$ and $(D', \phi')$, by the above we can assume that both stabilisations are of the same rank. If $(D, \phi)$ and $(D', \phi')$ have distance less than 1 with respect to the operator norm, orthogonal projection defines an isomorphism between their kernels. Note that since the space of such operators is convex, this open ball is contractible.\par 
                In general, we can pick a path in the space of such $(D, \phi)$, and iterate this process along the path. Because the spaces of such pairs $(D, \phi)$ (up to stabilisation) are weakly contractible, the uniqueness result follows from a similar argument. 
            \end{proof}
            \begin{RR}
                We needed compactness of $X$ for there to exist a family of linear maps $\phi_x$ as above. If $X$ was not compact, the $\hat{D}_x$ may not have a uniform upper bound on the rank of its cokernel, in which case such $\phi_x$ cannot exist.\par 
                If $X$ was not assumed to be compact, we could construct the index bundle on compact subsets $Y \subseteq X$ and take colimits in an appropriate manner. This would require a longer discussion and is not required for our purposes.
            \end{RR}
            % Given a real vector bundle $G \rightarrow X$, we can similarly define the index bundle using a stabilisation given by a continuous family of linear maps
            % $$\phi_x: G_x \rightarrow W^q\left(\Lambda^{0, 1} T^*D^2 \otimes E|_{D^2 \times \{ x \}}\right)$$
            % such that the 
            We will require some elementary properties of index bundles.
            \begin{LL}\label{Index Bundle Restriction}
                If $X'$ is another finite CW complex and $f: X' \rightarrow X$ is any map, then there is a natural isomorphism of virtual vector bundles
                $$f^* \Ind(E, F) \cong \Ind f^* (E, F)$$
                where we write $f^*(E, F)$ for $(Id_{D^2} \times f)^*(E, F)$.
            \end{LL}
            \begin{proof}
                All of the choices made to define the index bundle are compatible under pullbacks.
            \end{proof}
            \begin{LL}\label{Index Bundle Stability}
                If $(E, F)$ and $(E', F')$ are bundle pairs on $(D^2, \partial D^2) \times X$, then there is a natural isomorphism of virtual vector bundles $$\Ind(E, F) \oplus \Ind(E', F') \cong \Ind(E \oplus E', F \oplus F')$$
            \end{LL}
            \begin{proof}
                All of the choices made to define the index bundle are compatible under direct sums.
            \end{proof}
            \begin{CC}\label{ind is stable}
                A stable isomorphism between $(E, F)$ and $(E', F')$ induces a stable isomorphism between $\Ind(E, F)$ and $\Ind(E', F')$.
            \end{CC}
            \begin{LL}\label{Tensoring and Ind}
                Let $G \rightarrow X$ be a real virtual vector bundle over $X$. Then there is a canonical (up to weakly contractible choice) stable isomorphism
                $$\left(\Ind(E, F) \right) \otimes G \cong \Ind \left((E, F) \otimes G \right)$$
            \end{LL}
            \begin{proof}
                We assume $G$ is an actual vector bundle, and observe that the general case follows from Lemma \ref{Index Bundle Stability}.\par 
                Let $D$ be a family of Cauchy-Riemann operators on $(E, F)$, and $\phi$ a stabilisation of rank $N$, corresponding to the family of stabilised operators $T$. We pick another vector bundle $G'$ over $X$, along with an isomorphism $G \oplus G' \cong \R^M_X$ for some $M$.\par 
                For each $x$ in $X$, we consider the Cauchy-Riemann operator
                $$D_x \otimes Id_{G_x}: \Gamma\left(\left(E|_{D^2 \times \{x\}}, F|_{\partial D^2 \times \{x\}}\right)\otimes G_x\right) \rightarrow \Omega^{0, 1} \left( E_{D^2 \times \{x\}} \otimes G_x \right)$$
                along with the stabilisation of rank $NM$
                $$\theta_x: \R^{NM} \cong (\R^N \otimes G_x) \oplus (\R^N \otimes G_x') \rightarrow W^q\left(\Lambda^{0, 1}T^* D^2 \otimes E|_{D^2 \times \{x \}} \otimes G_x\right)$$
                sending $(a \otimes u, b \otimes v)$ to $\phi_x(a) \otimes u$.\par 
                Then $\Ind\left(\left(E, F\right) \otimes G\right)_x \oplus \R^{NM}$ is equal to $\Ker S_x$ by definition, where $S_x$ is the stabilised operator 
                $$\left(\widehat{D_x \otimes Id_{G_x}}\right) \oplus \theta_x$$
                But $\Ker S_x$ is isomorphic to
                $$\left(\left(\Ker T_x\right) \otimes G_x\right) \oplus \left(\R^N \otimes G'_x\right)$$
                which is isomorphic to
                $$\left(\left(\Ind(E, F)\right)_x \otimes G_x\right)  \oplus \R^{NM}$$
                A similar argument to Lemma \ref{Index Bundles Are Canonical} shows that this isomorphism is canonical up to weakly contractible choice.
            \end{proof}
            There is a standard bundle pair of rank 1
            $$H = (\mathbb{C}, \delta) \rightarrow (D^2, \partial D^2)$$
            where $\delta(z) = \sqrt{z} \mathbb{R}$ for $z$ in $\partial D^2$.\par 
            We can pick a trivialisation $E \cong \mathbb{C}^n_{D^2 \times \{x\}}$ over $D^2 \times \{x\}$ for a point $x$ in $X$. Then $F|_{\partial D^2 \times \{x\}}$ determines a loop in the space of totally real subspaces of $\mathbb{C}^n$, which is isomorphic to $U(n) / O(n)$. There are isomorphisms
            $$\pi_1 U(n) / O(n) \cong \mathbb{Z}$$
            for all $n$, compatibly with stabilisation in $n$. There are two such choices of isomorphism, and we fix the one such that $H \oplus (\mathbb{C}^{n-1}_{D^2}, \mathbb{R}^{n-1}_{\partial D^2})$ is sent to 1.
            \begin{DD}
                We define the \emph{Maslov index} of $(E, F)$, denoted $\mu(E, F)$, to be the image of $F$ under the isomorphism $\pi_1 U(n) / O(n) \rightarrow \mathbb{Z}$ after choosing a trivialisation of $E$ over a point in $X$. This is well-defined on each path component of $X$, i.e. it is a map $\pi_0 X \rightarrow \Z$.
            \end{DD}
            Using the index theorem, one can compute the virtual rank of the index bundle.
            \begin{LL}[{\cite[Theorem C.1.10]{McDuff-Salamon}}]\label{Virtual Dimension Formula}
                The virtual rank of $\Ind(E, F)$ is
                $$n + \mu(E, F)$$
            \end{LL}
            We note that therefore $\Ind\,H \cong \R^2$. We fix such an isomorphism (forever).
        \subsection{Orientations}\label{Orientations Section}
            If $(E, F)$ is a virtual bundle pair over $(D^2, \partial D^2) \times Y$ for some (possibly non-compact) space $Y$, for any map from a finite CW complex $X$ to $Y$, $\Ind(E, F)|_X$ is well-defined, and furthermore this is compatible with restriction.
            \begin{DD}
                % We say that $\Ind(E, F)$ is \emph{$R$-orientable} if $\Ind(E, F)|_X$ is $R$-orientable for all such $X$.\par
                An \emph{$R$-orientation} on $\Ind(E, F)$ is a choice of $R$-orientation on $\Ind(E, F)|_X$ for all $X$, compatible with restriction. We say $\Ind(E, F)$ is \emph{$R$-orientable} if such a choice exists.
            \end{DD}
            Our goal in this section will be to establish the following two propositions. These will be used later to find conditions under which the tangent bundles to the moduli spaces constructed in Sections \ref{Moduli Subsection} and \ref{Other Moduli Subsection} are $R$-orientable for various ring spectra $R$.
            \begin{PP}\label{Trivial Orientability}
                A stable trivialisation of $(E, F)$ induces a stable trivialisation of $\Ind(E, F)$, and hence an $R$-orientation for any ring spectrum $R$, compatibly with direct sums and pullbacks.
            \end{PP}
            \begin{proof}
                We write $(\C^k, \R^k)_X$ for the bundle pair $(\C^k_{D^2 \times X}, \R^k_{\partial D^2 \times X})$. By Corollary \ref{ind is stable}, it suffices to prove the result for $(E, F) = (\C^n, \R^n)_X$. Over every point $x$ in the base $X$, we pick the standard Cauchy-Riemann operator 
                $$\bar{\partial}:\Gamma\left(\C^n, \R^n\right) \rightarrow \Omega^{0, 1}\left( \C^n\right)$$
                on sections over the disc $D^2 \times \{x\}$.\par 
                Evaluation at any fixed point $z$ in $\partial D^2$ defines a map $\Ker \bar{\partial} \rightarrow \R^n$, which is an isomorphism (and which does not depend on $z$). Then because the index of this Cauchy-Riemann operator is $n$, (the Sobolev completion of) this Cauchy-Riemann operator must be surjective. Therefore $\Ker \bar{\partial} = \Ind(\C^n, \R^n)_X$ and so this gives us the required trivialisation of $\Ind(E, F)$.
            \end{proof}
            The following proposition seems well-known to experts, but to the author's knowledge a proof has not yet appeared in the literature.
            \begin{PP}\label{Complex Orientability}
                A stable trivialisation of $F$ induces a stable complex structure on $\Ind(E, F)$, compatibly with direct sums and pullbacks.
            \end{PP}
            \begin{proof}
                Stabilising if necessary, we can assume that we have a choice of (unstable) trivialisation $F \cong \R^n_{\partial D^2 \times X}$.\par 
                This trivialisation of $F$ induces a trivialisation of $E|_{\{1\} \times X}$, which (since the disc is contractible) naturally extends to a trivialisation $E \cong \C^n_{D^2 \times X}$ over $D^2 \times X$. This is compatible with the trivialisation of $F$ over $\{1\} \times X$, but not necessarily over the rest of $\partial D^2 \times X$. \par 
                By identifying $\partial D^2$ with $S^1$, we identify $\partial D^2 \times X / (\{1\} \times X)$ with the suspension $\Sigma X_+$. There is then an induced map $f: \Sigma X_+ \rightarrow U/O$, sending $y$ to the the totally real subspace $F_y$ of $E_y \cong \C^n$. Here $U/O$ is naturally identified with the colimit as $N \rightarrow \infty$ of totally real subspaces of $\C^N$, where a matrix $A$ in $U(N)/O(N)$ corresponds to the totally real subspace $A(\R^N) \subseteq \C^N$. \par 
                The group of based homotopy classes of maps $\Sigma X_+ \rightarrow U/O$ classifies virtual bundle pairs over $(D^2, \partial D^2) \times X$ equipped with a stable trivialisation over $\{1\} \times X$, up to stable isomorphism and stabilisation (i.e. direct summing a copy of $(\C, \R)_X$).
                \begin{MM}
                    The trivialisation of $F$ induces a lift $\tilde{f}: \Sigma X_+ \rightarrow U$, fitting into the commutative diagram
                    \[\begin{tikzcd}
                        & U \arrow[d, "q"] \\
                        \Sigma X_+ \arrow[r, "f"] \arrow[ru, "\tilde{f}"] & U/O
                    \end{tikzcd}\]
                    where $q: U \rightarrow U/O$ is the quotient map.
                \end{MM}
                \begin{proof}[Proof of claim]
                    We will construct $\tilde{f}: \Sigma X_+ \rightarrow U(n)$, then stabilising in $n$ will give the desired map.\par 
                    Let $\phi: \R^n_{\partial D^2 \times X} \xrightarrow{\cong} F \subseteq \C^n_{\partial D^2 \times X}$ be the  trivialisation (which we assume to be orthogonal), let $a_1, \ldots, a_n$ be the standard basis of $\C^n$ and let $b_1, \ldots, b_n$ be the standard basis of $\R^n$. \par 
                    We then define $\tilde{f}(y)$ to be the unitary matrix that sends $a_i$ to $\phi(b_i)$ for all $i$.
                \end{proof}
                There is a natural commutative diagram, from Bott periodicity (using the isomorphisms in \cite{AtiyahK-TheoryandReality}):
                \[\begin{tikzcd}
                    \left[X_+, BU \times \Z\right] \arrow[r] \arrow[d, "g"', "\cong"] & \left[X_+, BO \times \Z\right] \arrow[d, "h"', "\cong"]\\
                    \left[\Sigma X_+, U\right] \arrow[r, "q_*"] & \left[\Sigma X_+, U/O\right]
                \end{tikzcd}\]
                $f$ lives in the bottom right group and $\tilde{f}$ lives in the bottom left group of the above diagram.\par 
                There is a map 
                $$I: \left[\Sigma X_+, U/O \right] \rightarrow \left[X_+, BO \times \Z\right]$$
                sending a map $r$ to the index bundle of the bundle pair determined by $r$.
                \begin{MM} 
                    The composition $I \circ h: [X_+, BO \times \Z] \rightarrow [X_+, BO \times \Z]$ is the identity map.
                \end{MM}
                In particular, this implies that the map $I$ is an isomorphism. This fact, along with its relation to orientations on moduli spaces of holomorphic curves, was first observed by de Silva in \cite{deSilva}.
                \begin{proof}[Proof of claim]
                    Let $H$ be the bundle pair of rank 1 over $(D^2, \partial D^2)$ defined in Section \ref{IndexBundleSection}. It admits a canonical trivialisation over $\{1\}$. Then $h$ is defined to send (the classifying map for) a virtual vector bundle $G$ over $X$ to (the classifying map for) the bundle pair 
                    $$H \otimes G - (\C, \R)_X \otimes G$$
                    equipped with its canonical stable trivialisation over $\{1\} \times X$.\par 
                    Then by the properties of index bundles developed in Section \ref{IndexBundleSection}, we have isomorphisms
                    \begin{align*}
                        \Ind(h(G)) & \cong \Ind\left(H \otimes G - (\C, \R)_X \otimes G\right)\\
                        & \cong \Ind(H) \otimes G - \Ind(\C, \R)_X \otimes G\\
                        & \cong \R^2_X \otimes G - \R_X \otimes G\\
                        & \cong G
                    \end{align*}
                    as required.
                \end{proof}
                Then since $h^{-1}(f)$ is the classifying map of $\Ind(E, F)$, $g^{-1}(\tilde{f})$ is the choice of stable complex structure that we desired. Note that all choices we made were compatible with direct sums and pullbacks.
            \end{proof}
    \section{The moduli spaces $\mathcal{P}
    $ and $\sW$}\label{Moduli Section}
        \subsection{Construction of the moduli spaces}\label{Moduli Subsection}
            Let $C$ be a convex domain in the upper half plane in $\mathbb{C}$, with smooth boundary $\partial C$ containing 0. Let $f: \partial C \rightarrow [0,1]$ be a smooth map sending $0$ to $0$, and let $J$ be an $\omega$-tame almost complex structure on $M$ which is convex at infinity.\par
            \begin{DD}
                We define $\sD_{f, C}$ to be the space of smooth maps $u: C \rightarrow M$ such that for all $z$ in $\partial C$, $u(z)$ lies in $\psi^{f(z)}(L)$.\par 
                We define $\pi$ to be the natural evaluation map $\pi: \sD_{f, C}(J) \rightarrow L$ sending $u$ to $u(0)$.
            \end{DD}
            % Given $u$ in $\sD_f$, we let $[\partial u]$ in $\pi_0 \sL L$ be the homotopy class of the map $\partial C \rightarrow L$, which sends $z$ in $\partial C$ to $\left(\psi^{f(z)}\right)^{-1}(u(z))$. We let $\sD^0_f$ denote the subspace of $\sD_f$ consisting of $u$ in $\sD_f$ such that $[\partial u]$ is nullhomotopic. Then for any $f: X \rightarrow \sD^0_f$, $\Ind \rightarrow X$ has virtual rank $n$.
            \begin{DD}
                We define $\sU_{f, C}(J)$ to be the space of maps $u$ in $\sD_{f, C}$ which are $J$-holomorphic, namely
                $$J \circ du = du \circ j$$
                where $j$ is the complex structure on $C \subseteq \C$.\par 
                We call the triple $(C, f, J)$ the \emph{moduli data}.
            \end{DD}
            Let $G$ be a fixed convex domain in $\mathbb{C}$ with smooth boundary, such that both line segments $(-\eta, \eta) \subset \R$ and $i + (-\eta, \eta) \subset i + \R$ are contained in $\partial G$ for some $\eta > 0$. Let $G_\pm$ be $G \cap \mathbb{C}_{\pm \re \geq 0}$.\par 
            For $l \geq 0$, define $Z_l$ to be $[0, 1]i + [-l, l]$, and $G_l$ to be 
            $${Z_l \cup (G_+ + l) \cup (G_- - l)}$$
            as shown below.\par 
            
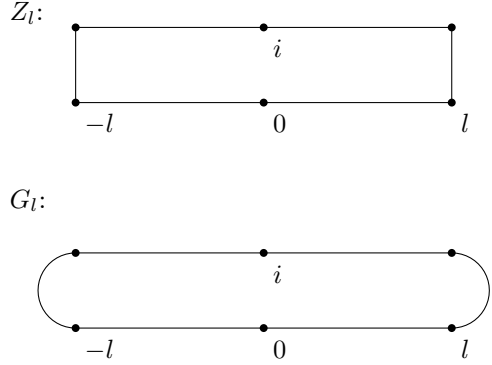
{\centering\begin{tikzpicture}[line cap=round,line join=round,>=triangle 45,x=1.0cm,y=1.0cm]
    \draw (0,0)-- (5,0);
    \draw (0,1)-- (5,1);
    \draw (5,1)-- (5,0);
    \draw (0,1)-- (0,0);            
    \draw (5,-3)-- (0,-3);
    \draw (0,-2)-- (5,-2);
    \draw (2.5,0.5) node[anchor=north west] {\parbox{0 cm}{$$ 0 $$}};
    \draw (2.5,1.5) node[anchor=north west] {\parbox{0 cm}{$$ i $$}};
    \draw (5,0.5) node[anchor=north west] {\parbox{0 cm}{$$ l $$}};
    \draw (0,0.5) node[anchor=north west] {\parbox{0 cm}{$$ -l $$}};
    \draw (-1,2) node[anchor=north west] {\parbox{0 cm}{$$ Z_l\mathchar\numexpr"6000+`:\relax $$}};
    \draw (2.5,-2.5) node[anchor=north west] {\parbox{0 cm}{$$ 0 $$}};
    \draw (2.5,-1.5) node[anchor=north west] {\parbox{0 cm}{$$ i $$}};
    \draw (5,-2.5) node[anchor=north west] {\parbox{0 cm}{$$ l $$}};
    \draw (0,-2.5) node[anchor=north west] {\parbox{0 cm}{$$ -l $$}};
    \draw (-1,-0.5) node[anchor=north west] {\parbox{0 cm}{$$ G_l\mathchar\numexpr"6000+`:\relax  $$}};
    \draw [shift={(0,-2.5)}] plot[domain=1.57:4.71,variable=\t]({1*0.5*cos(\t r)+0*0.5*sin(\t r)},{0*0.5*cos(\t r)+1*0.5*sin(\t r)});
    \draw [shift={(5,-2.5)}] plot[domain=-1.57:1.57,variable=\t]({1*0.5*cos(\t r)+0*0.5*sin(\t r)},{0*0.5*cos(\t r)+1*0.5*sin(\t r)});
    \begin{scriptsize}
        \fill [color=black] (0,0) circle (1.5pt);
        \fill [color=black] (5,0) circle (1.5pt);
        \fill [color=black] (5,1) circle (1.5pt);
        \fill [color=black] (0,1) circle (1.5pt);
        \fill [color=black] (2.5,0) circle (1.5pt);
        \fill [color=black] (0,-2) circle (1.5pt);
        \fill [color=black] (5,-2) circle (1.5pt);
        \fill [color=black] (0,-3) circle (1.5pt);
        \fill [color=black] (5,-3) circle (1.5pt);
        \fill [color=black] (2.5,1) circle (1.5pt);
        \fill [color=black] (2.5,-2) circle (1.5pt);
        \fill [color=black] (2.5,-3) circle (1.5pt);
    \end{scriptsize}
\end{tikzpicture}\par }
            We now define a 1-parameter family of moduli data $(C_l, f_l, J_l)$ as follows.
            $$
                C_l := 
                \begin{cases}
                    G_{l - 1} & \text{ if } l \geq 1\\
                    G_0 & \text{ if } l \leq 1
                \end{cases}
            $$
            $$
                f_l(z) := 
                \begin{cases}
                    \im z & \textit{ if } l \geq 1\\
                    l \im z & \textit{ if } l \leq 1
                \end{cases}
            $$
            and $J_l$ is defined to be some choice of $\omega$-tame almost complex structure (convex at infinity if $M$ is non-compact) on $M$ which is independent of $l$ (though we will apply an $l$-dependent generic perturbation to $J_l$ soon).\par 
            As a 1-parameter family of domains in $\C$ and maps to $[0, 1]$, $C_l$ and $f_l$ are continuous everywhere in $l$ and vary smoothly in $l$ except at 1, so we pick a small smoothing of these families supported near $l = 1$, by reparametrising in the $l$ direction.\par 
            Relative exactness of $L$ in $M$ gives us a uniform energy bound.
            % , the proof of \cite[Lemma 2]{Hofer} shows that
            \begin{LL}\label{Uniform Energy Bound}
                There exists some $A$ in $\R$ such that for any $l \geq 0$ and any smooth map $u$ in $\sD_{f_l, C_l}$, the topological energy $\int_{C_l} u^* \omega$ is bounded above by $A$.
            \end{LL}
            The proof is a minor adaptation of that of \cite[Lemma 2]{Hofer}. We remark that if the isotopy $\psi^t$ were through symplectomorphisms instead of Hamiltonian diffeomorphisms, the proof of this lemma would fail.
            \begin{proof}

                We choose $H: M \times [0,1] \rightarrow M$ a time-dependent Hamiltonian generating $\psi^t$, meaning that for $x$ in $M$ and $t$ in $[0,1]$,
                $$\omega\left(\cdot, \frac{d}{dt}(\psi^t)(x)\right) = dH_t(x)$$
                and we choose $H$ to be constant outside compact subsets of $M$.\par 
                Choose $l \geq 0$. There is some $\lambda$ in $[0,1]$ such that $f_l(z) = \lambda \im z$ for all $z$ in $\partial C_l$. Choose $u$ in $\sD_{f_l, C_l}$, and define $w: C_l \rightarrow M$ by
                $$w(s+it) := (\psi^{\lambda t})^{-1}(u(s+it))$$
                Then $w$ sends $\partial C_l$ to $L$ so by relative exactness of $L$, we must have $\int_{C_l} w^* \omega = 0$. \par 
                The topological energy of $u$ is
                \begin{align*}
                    \int_{C_l} u^* \omega &
                    = \int\limits_{s+it \in C_l} \omega(\partial_s u, \partial_t u) ds \, dt\\
                    & = \int\limits_{s+it \in C_l} \omega\left(d\psi^{\lambda t}(\partial_s w), d\psi^{\lambda t}(\partial_t w) + \frac{d}{dt}(\psi^{\lambda t})(w)\right) ds\,dt\\
                    & = \int\limits_{s+it\in C_l} \omega(\partial_s w, \partial_t w) ds\,dt + \int\limits_{s+it \in C_l} \omega\left(d\psi^t(\partial_s w), \lambda \frac{d}{dt}(\psi^t)(w)\right) ds\,dt \\
                    & = \int_{C_l} w^* \omega + \lambda \int\limits_{s+it \in C_l} dH_t\left(d\psi^t(\partial_s w)\right) ds\,dt\\
                    & = 0 + \lambda \int_{t=0}^1 H_t\left(\psi^t(w(s_{max}(t)+it))\right) - H_t\left(\psi^t(w(s_{min}(t)+it))\right) dt\\
                    & \leq \max_{t,x} H_t(x) - \min\limits_{t, x} H_t(x) =: A
                    % [X \wedge R, Y \wedge R]^R & \cong [X, Y \wedge R]\\
                    % & \cong [Y^\vee \wedge X, R]\\
                    % & \cong [Y^\vee, X^\vee \wedge R]\\
                    % & \cong [Y^\vee \wedge R, X^\vee \wedge R]^R.
                \end{align*}
                where $s_{max}(t)$ is the largest $s$ such that $s + it$ lies in $C_l$, and $s_{min}(t)$ is the smalles $s$ such that $s + it$ lies in $C_l$. 
            \end{proof}
            We now pick a Riemannian metric on $L$, with injectivity radius $\varepsilon$ and distance function $d$. Using this energy bound, \cite[Proposition 3]{Hofer} then directly implies the following result, which should be viewed as a form of Gromov compactness. Roughly, Hofer shows that for large $l$, these moduli spaces, when restricted to $Z_l$, live close to holomorphic strips with boundary on $L$ and $\psi^1(L)$, which we know are constant since $L = \psi^1(L)$ is relatively exact.
            \begin{LL}\label{Hofer Bound}
                There is some $c > 1$ such that for all $l \geq c$ and $u$ in $\sU_{f_l, C_l}(J_l)$, 
                $$d(u(i), u(0)) < \varepsilon$$
            \end{LL}
            \begin{proof}
                We apply \cite[Proposition 3]{Hofer}, to the Lagrangians $L$ and itself. In this setting, for $l > 1$, what Hofer calls $\Omega_{J_l}(L, L)$ is the space of finite-energy $J_l$-holomorphic strips with boundary sent to $L$, and so consists only of constant maps to $L$, and in particular does not depend on $l$.\par
                Then take the open neighbourhood $U$ of $\Omega_{J_l}(L, L)$ to be the space of maps $u: \R \times [0,1] \rightarrow M$ such that $d(u(i), u(0)) < \varepsilon$ (where we equip these mapping spaces with the weak $C^\infty$ Whitney topology). \cite[Proposition 3]{Hofer} says that there is some $c > 1$ such that for all $l \geq c$ and $u$ in $\sU_{f_l, C_l}(J_l)$, $r_l(u)$ lies in $U$, where Hofer defines $r_l(u): \R \times [0,1] \rightarrow M$ so that in particular $r_l(u)$ and $u$ agree when restricted to neighbourhoods of $i[0, 1]$, and so in particular $u$ also lies in $U$.
            \end{proof}
            Standard techniques in \cite{McDuff-Salamon} show that for a generic ($l$-dependent) perturbation of $J_l$, $\mathcal{W}' := \bigsqcup\limits_{l \geq 0} \sU_{f_l, C_l}(J_l)$ is a smooth manifold with boundary $\sU_{f_0, C_0}(J_0)$, which consists only of constant maps $C_0 \rightarrow L$ and hence $\pi: \sU_{f_0, C_0}(J_0) \rightarrow L$ is a diffeomorphism. If we choose the perturbation small enough then Lemma \ref{Hofer Bound} still holds, by Lemma \ref{Uniform Energy Bound} along with Gromov compactness.\par 
            Let $\sU_l = \sU_{f_l, C_l}(J_l)$, and let $p: \sW' \rightarrow \R$ send $\sU_l$ to $l$.
            \begin{LL}
                Each $\sU_l$ is compact and $p$ is a proper map.
            \end{LL}
            \begin{proof}
            
                We will show each $\sU_l$ is compact, a similar argument shows that $p$ is a proper map.\par  
                Let $u_i$ be a sequence in $\sU_l$. Gromov compactness as in \cite{Frauenfelder-Zehmisch} implies that this contains a convergent subsequence, but only after precomposing with a sequence of holomorphic automorphisms of the domain.\par 
                Let $\tilde{M}$ be $M \times \C$, with almost complex structure $J_l$ times the standard complex structure on $\C$. Let $\tau: \tilde{M} \rightarrow \C$ be the projection map onto the second factor. Let $\tilde{L}$ be 
                $$\bigcup\limits_{z \in \partial C_l} \psi^{f_l(z)}(L) \times \{z\}$$
                a totally real submanifold of $\tilde{M}$.\par 
                % We let $B \subseteq \pi_2(\tilde{M}, \tilde{L})$ be the set of relative homotopy classes of maps $C_l \rightarrow \tilde{M}$ which are sections of the projection map to $\C$.\par 
                We let $\tilde{\sU}_l$ be the moduli space of holomorphic discs in $\tilde{M}$ with boundary in $\tilde{L}$. %living in relative homotopy classes in $B$.\par 
                There is an embedding $\theta: \sU_l \hookrightarrow \tilde{\sU}_l$ sending $u$ to the map $\tilde{u}$ which sends $z$ to $(u(z), z)$. This clearly is a bijection with the subset of $\tilde{\sU}_l$ consisting of sections of $\tau|_{\tau^{-1}C_l}$.\par  
                For $u$ in $\sU_l$, the energy of $\tilde{u}$ is equal to the energy of $u$ plus the area of $C_l$, so in particular by Lemma \ref{Uniform Energy Bound}, the energies of the $\tilde{u}_i$ are uniformly bounded. Therefore by \cite[Theorem 1.1]{Frauenfelder-Zehmisch}, there is a subsequence (which we will continue to call $\tilde{u_i}$, by abuse of notation) and a sequence of holomorphic automorphisms $\phi_i$ of $C_l$, such that $\tilde{u}_i \circ \phi_i$ lives in $\tilde{\sU}_l$ for all $i$, and $\tilde{u}_i \circ \phi_i$ converges to some $v$ in $\tilde{\sU}_l$.\par 
                $\tau \circ \tilde{u}_i \circ \phi_i = \phi_i$ also converges to $\phi := \tau \circ v$ which is a holomorphic automorphism of $C_l$, so $\tilde{u}_i$ converges to $v \circ \phi^{-1}$, which lies in the image of $\theta$ and so $v = \tilde{u}$ for some $u$ in $\sU_l$. Therefore $u_i$ converges to $u$ in $\sU_l$.
            \end{proof}
            We pick $c' > c$ (where $c$ is from Lemma \ref{Hofer Bound}), a regular value of $p$. We then define $\sW$ to be the path component of $\sU_0 = L$ in $p^{-1}[0, c']$, and $\sP$ to be 
            $$\sW \cap \sU_{c'}$$
            Then, by construction, $\pi: \mathcal{W} \rightarrow L$ is a cobordism from $Id: L \rightarrow L$ to $\pi: \mathcal{P} \rightarrow L$ over $L$. Furthermore whilst $\sW'$ may have had different path components of different dimensions, $\sW$ is a manifold (with boundary) of dimension $n + 1$, and $\sP$ is a manifold of dimension $n$.
            \begin{proof}[Proof of Lemma \ref{Moduli Space Works}]
                We define $\pi': \sP \rightarrow L$ to send $u$ to $u(i)$. Then by Lemma \ref{Hofer Bound}, this is homotopic to $\pi$. We now write down a homotopy 
                $$H: \sP \times [0, 1] \rightarrow L$$
                from $\psi^1 \circ \pi$ to $\pi'$.\par 
                Let $\gamma: [0, 1] \rightarrow \partial C_{c'}$ be some fixed path from $0$ to $i$. For $u$ in $\sP$ and $t$ in $[0, 1]$, we define $H(u, t)$ to be
                
                $$\psi^1\left(\left(\psi^{\im \gamma(t)}\right)^{-1}\left(u\left(\gamma(t)\right)\right)\right)$$
                
            \end{proof}
            \begin{RR}
                \cite[Theorem 1]{Hofer} shows that even without any genericity assumption on the almost complex structure, the map $\pi: \sP \rightarrow L$ induces an injective map 
                $$H^*(L; \Z/2) \cong \check{H}^*(L; \Z/2) \rightarrow \check{H}^*(\sP; \Z/2)$$
                where $\check{H}^*$ denotes \v{C}ech cohomology. The statement of Lemma \ref{Hofer Bound} still holds and hence so does Lemma \ref{Moduli Space Works}. From this it follows that $\psi^1|_L$ induces the identity map on $H^*(L; \Z/2)$.\par 
                This can be adapted to other generalised cohomology theories too, as in \cite{Me+Amanda}.
            \end{RR}
        \subsection{Orientations on the moduli space}\label{Orienty Subsection}
            In this subsection, we will prove Proposition \ref{Orientation Criteria}.\par 
            We define $\sD_0$ to be the space of triples $(f, C, u)$, where $C$ is a convex domain in the upper half plane in $\C$ with smooth boundary $\partial C$ containing 0, ${f: \partial C \rightarrow [0, 1]}$ is a smooth map sending 0 to 0, and $u$ is an element of $\sD_{f, C}$. Note that because the choices of $C$ and $f$ were contractible, each inclusion $\sD_{f, C} \hookrightarrow \sD_0$ is a homotopy equivalence.\par 
            There is a tautological smooth fibre bundle over $\sD_0$ with fibre over $(f, C, u)$ given by $C$. The structure group of this bundle is the group of orientation-preserving diffeomorphisms of $D^2$ which send 1 to itself. This group is contractible, so we can pick a trivialisation $\Phi$ from $D^2 \times \sD_0$ to this bundle, sending $1$ in $D^2$ to $0$ in each $C$. Note the space of such choices is contractible.\par 
            For each $(f, C)$, we choose some extension $\tilde{f}$ of $f$ to the entirety of $C$. Since for each $(f, C)$ the space of such extensions is contractible, we can make some continuous choice over the whole of $\sD_0$, and furthermore this choice is unique up to contractible choice.\par 
            There is then an evaluation map 
            $$ev: \left(D^2, \partial D^2\right) \times \sD_0 \rightarrow (M, L)$$
            sending $(z, u)$ in $D^2 \times \sD_{f, C}$ to 
            $$\left(\psi^{\tilde{f}(\Phi(z))}\right)^{-1}\left(u\left(\Phi(z\right))\right)$$
            Because the space of choices for the extensions $\tilde{f}$ was contractible, this defines $ev$ uniquely up to a contractible space of choices.\par 
            There is a bundle pair over $\left(D^2, \partial D^2 \right) \times \sD_0$ given by $ev^*\left(TM, TL\right)$. For a map from a finite CW complex $g: X \rightarrow \sD_0$, we write $\Ind$ for the index of the pullback of this bundle pair along $g$. Note that the virtual rank of $\Ind$ may be different on different components of $X$ if the evaluation map lands in different components of $\sD_0$.\par 
        
            % For any moduli data $(C, f, J)$, there is an evaluation map $ev_{f, C}: D^2 \times \sD_{f, C} \rightarrow M$ sending $(z, u)$ in $\sD_{f, C}$ to $(\psi^{\tilde{f}}(z))^{-1}(u(z))$, where $\tilde{f}$ is some smooth extension of $f$ to the entirety of $C$. These restrict to evaluation maps $ev_{f, C}: \partial D^2 \times \sD_{f, C} \rightarrow L$. Note that these are canonical up to a contractible space of choices.\par 
            % There is a bundle pair over $(D^2, \partial D^2) \times \sD_{f, C}$ given by $(ev_{f, C}^* TM, ev_{f, C}^* TL)$. Given any map from a finite CW complex $g: X \rightarrow \sD_{f, C}$, we write $\Ind_{f, C}$ for the index bundle of the pullback bundle pair $(g^* ev_{f, C}^* TM, g^* ev_{f, C}^* TL)$. This is compatible with composition of such maps. Note that the virtual rank of $\Ind_{f, C}$ may be different on different connected components of $X$ if they land in different connected components of $\sD_{f, C}$.\par
            Because the linearisation of the Cauchy-Riemann equation along its zero set is a Cauchy-Riemann operator, if $J$ is regular, $T\sU_{f, C}(J)$ is stably isomorphic to the index bundle $\Ind$, and similarly $T\sW$ is stably isomorphic to $\Ind \oplus \R$.\par
            Fix $X$ a finite CW complex and $g: X \rightarrow \sD_0$ some map. In the rest of this section we will find conditions under which $\Ind$ is $R$-orientable over $X$ for various ring spectra $R$.
            \begin{PP}
                % If there is a stable trivialisation $TL \cong \R^n$ which (after applying $\cdot \otimes \C$) extends to a stable trivialisation $TM \cong \C^n$, then $\Ind - \pi^*TL$ is stably trivial and hence $R$-orientable for any ring spectrum $R$.
                If there is a real vector bundle $E$ over $M$, along with stable isomorphisms 
                $TM \cong E \otimes \C$ and $TL \cong E|_L$ compatible with each other, then $\Ind-\pi^*TL$ is stably trivial and hence $R$-orientable for any ring spectrum $R$.
            \end{PP}
            
            \begin{proof}
                We observe that the bundle pair $ev^*(TM, TL) - ev^*(E \otimes \C, E)$ is stably trivial, and that $ev^*(E \otimes \C, E) \cong q^*\pi^*(E \otimes \C, E)$ where $q: (D^2, \partial D^2) \times X \rightarrow X$ collapses the disc to a point. Therefore $\Ind(ev^*(E \otimes \C, E)) \cong \pi^*E \cong \pi^* TL$. \par 
                Then $\Ind - \pi^* TL$ is stably isomorphic to $\Ind(ev^*(TM, TL) - ev^*(E \otimes \C, E))$, which is the index bundle of a stably trivial bundle pair, which is stably trivial by Proposition \ref{Trivial Orientability}.
            \end{proof}
            \begin{PP}
                If there is a stable trivialisation $TL \cong \R^n_{L_3}$ over the 3-skeleton of $L$ which (after applying $\cdot \otimes \C$) extends to a stable trivialisation $TM \cong \C^n_{M_4}$ over the 4-skeleton of $M$, then $\Ind$ and hence $\Ind - \pi^*TL$ are $KO$-orientable.
            \end{PP}
            \begin{proof}
                We pick a 3-skeleton $L_3$ of $L$, a 4-skeleton $M_4$ of $M$ containing $L_3$, and a 2-skeleton $X_2$ of $X$. Then the map of pairs $ev: (D^2, \partial D^2) \times X_2 \rightarrow (M, L)$ can be homotoped to land in $(M_4, L_3)$, and we can apply Proposition \ref{Trivial Orientability} to see that $\Ind$ admits a stable trivialisation over a 2-skeleton. Then by \cite[Theorem 12.3]{Atiyah-Bott-Shapiro}, $\Ind$ is $KO$-orientable.
            \end{proof}
            \begin{LL}\label{Complex Structures Skeleta}
                If there is a stable trivialisation $TL \cong \R^n_{L_{i+1}}$ over the $(i + 1)$-skeleton of $L$, $\Ind$ admits a stable complex structure over the $i$-skeleton of $X$.
            \end{LL}
            \begin{proof}
                The real part of the bundle pair $(ev^* TM, ev^* TL) \rightarrow (D^2, \partial D^2) \times X$ is pulled back from a map $\partial D^2 \times X \rightarrow L$. Letting $X_i$ be some choice of $i$-skeleton of $X$, $ev: \partial D^2 \times X_i\rightarrow L$ can be homotoped to have image in the $(i+1)$-skeleton of $L$, so by Proposition \ref{Complex Orientability}, $\Ind$ admits a stable complex structure over $X_i$.
            \end{proof}
            \begin{CC}\label{aaaa}
                \begin{enumerate}
                    \item If $L$ admits a spin structure, $\Ind$ and $\Ind - \pi^* TL$ are orientable.
                    \item If $L$ admits a spin structure, $\Ind$ and $\Ind - \pi^* TL$ are $KU$-orientable.
                \end{enumerate}
            \end{CC}
            \begin{proof}
                \begin{enumerate}
                    \item If $L$ admits a spin structure, $TL$ is stably trivial over the 2-skeleton of $L$, and so $\Ind$ admits a complex structure over the 1-skeleton of $X$ and hence is orientable.
                    \item A vector bundle over $L$ is stably trivial over an $i$-skeleton $L_i$ iff the classifying map $L \rightarrow BO$ is nullhomotopic when restricted to $L_i$. Since $\pi_3 BO = 0$, if $L_2 \rightarrow BO$ is nullhomotopic, so is $L_3 \rightarrow BO$. So if $L$ admits a spin structure, $TL$ admits a stable trivialisation over any 3-skeleton. Therefore $\Ind$ admits a spin\textsuperscript{c} structure, i.e. it admits a stable complex structure over any 2-skeleton. Then by \cite[Theorem 12.3]{Atiyah-Bott-Shapiro}, $\Ind$ is $KU$-orientable.
                \end{enumerate}
            \end{proof}
            We use the computation of $w_1$ of the index bundle of a bundle pair from \cite{Georgieva} to weaken the hypothesis of Corollary \ref{aaaa}(1).
            \begin{LL}
                If $w_2(L)(x) = 0$ whenever $x$ is a homology class represented by a 2-torus $S^1 \times \partial D^2$ in $L$ which extends to a solid torus $S^1 \times D^2$ in $M$, then $\Ind - \pi^* TL$ is orientable.
            \end{LL}
            \begin{proof}
                $w_1(\Ind - \pi^*TL)$ vanishes by \cite[Theorem 1.1]{Georgieva}.
            \end{proof}
            Together, the results in this section combine to form Proposition \ref{Orientation Criteria}.
    \section{Monodromy in the looped case}\label{Looped Section}
        \subsection{General set-up}
            We define $L_{\psi^1}\subseteq S^1 \times M$ to be the space of pairs $(t, x)$, where $t$ lies in $S^1 = \R / \Z$ and $x$ lies in $\psi^t(L)$. We let $q: L_{\psi^1} \rightarrow S^1$ be the projection map to the first co-ordinate. This is a model for the mapping torus of $\psi^1|_L$. We denote its vertical tangent space by $T^vL_{\psi^1}$. \par 
            The Hamiltonian flow induces a natural family of maps which we denote by $\Psi^t: L_{\psi^1} \rightarrow L_{\psi^1}$ for $t$ in $\R$, which live over the identity map on $S^1$ for $t$ in $\Z$. More explicitly, pick $s$ in $S^1$, $x$ in $\psi^s(L)$ and $t$ in $\R$, and write $t + s = m + \tau$, where $m$ is in $\Z$ and $\tau$ is in $[0, 1)$. Then we define $\Psi^t(s, x)$ to be 
            $$\left(\tau, \psi^\tau \circ (\psi^1)^{\circ m} \circ (\psi^s)^{-1}(x)\right)$$
            % $$\left(\tau, \psi^\tau\left(\left(\psi^1\right)^m\left(\left(\psi^s\right)^{-1}\left(x\right)\right)\right)\right)$$
            Then $\Psi^0$ is the identity map, $\Psi^t$ is continuous in $t$ and $\Psi^t \circ \Psi^s = \Psi^{t + s}$ for all $t, s$ in $\R$. Furthermore $\Psi^1|_{q^{-1}(\{0\})} = \psi^1|_L$. $\Psi^t$ should be thought of as parallel transporting by a distance $t$ around the base of the fibre bundle $L_{\psi^1} \twoheadrightarrow S^1$.\par 
            The free loop space $\sL L_{\psi^1}$ of $L_{\psi^1}$ is naturally a fibre bundle over $L_{\psi^1}$. We define $\sL$ to be the restriction of this to the fibre over $0$ in $S^1$. Then $\sL$ is naturally a fibre bundle over $L$ and splits into a disjoint union
            $$\sL = \bigsqcup\limits_{j \in \Z} \sL_j$$
            where $\sL_j$ is the space of loops in $\sL$ with winding number $j$ about $S^1$. \par 
            $\Psi^1$ induces a map $\sL \rightarrow \sL$ which preserves winding numbers, $\Psi^1(\sL_i) = \sL_i$ for all $i$, and we will also write $\Psi^1$ for this restriction.\par 
            We define $\sL'_j$ to be subspace of $\sL_j$ given by loops $\gamma$ such that the composition $S^1 \rightarrow L_{\psi^1} \rightarrow S^1$ is given by $z \mapsto j z$, and $\sL'$ to be the disjoint union of all $\sL'_j$. Then $\Psi^1$ sends $\sL'$ to $\sL'$.
            \begin{LL}\label{Reparametrisation}
                The inclusion $\sL' \hookrightarrow \sL$ is a homotopy equivalence.
            \end{LL}
            \begin{proof}
                We first define a continuous map $f: \sL_j \rightarrow \left\{\mathrm{continuous\,maps}\, [0, 1] \rightarrow \R\right\}$. For all $\gamma$ in $\sL_j$, $f(\gamma)$ is uniquely determined by the following properties.
                \begin{enumerate}
                    \item $f(\gamma)(0) = 0$
                    \item $f(\gamma)(1) = j$
                    \item The composition 
                    $$[0, 1] \rightarrow S^1 \xrightarrow{\gamma} L_{\psi^1} \xrightarrow{q} S^1$$
                    agrees with the composition
                    $$[0, 1] \xrightarrow{f(\gamma)} \R \rightarrow S^1$$
                \end{enumerate}
                We define a map $F: \sL \rightarrow \sL'$ by
                $$F(\gamma)(t) = \Psi^{jt - f(\gamma)(t)}(\gamma(t))$$
                for $t$ in $[0, 1]$. Then this is a homotopy inverse to the inclusion $\sL' \hookrightarrow \sL$.
            \end{proof}
            By Lemma \ref{Reparametrisation}, $\sL'_0 \xrightarrow{\simeq} \sL_0$, and we see that the action of $\Psi^1$ on $\sL_0$ restricts to the action of $\psi^1|_L$ on $\sL L = \sL'_0$.
            % By Lemma \ref{Reparametrisation}, $\sL_0 \simeq \sL L$ and it is clear that the action of $\Psi^1$ on $\sL_0$ is the same as the action of $\psi^1$ on $\sL L$. 
            Therefore to study the action of $\psi^1|_L$ on $\sL L$, it suffices to study the action of $\Psi^1$ on $\sL_0$.

            \begin{LL}\label{Monodromy is trivial on L1}
                $\Psi^1$ acts as the identity on $\sL_{\pm1}$ up to homotopy.
            \end{LL}
            \begin{proof}
                Note by Lemma \ref{Reparametrisation} it suffices to prove $\Psi^1$ acts as the identity on $\sL'_{\pm1}$.\par 
                We define $H: \sL'_{\pm1} \times [0, 1] \rightarrow \sL'_{\pm1}$ by
                $$H(\gamma, s)(t) = 
                \Psi^{\pm s }(\gamma(t - s))$$
                Then $H(\cdot, 0)$ is the identity and $H(\cdot, 1)$ is $\Psi^{\pm1}$. In the case of $+1$, $H$ is the required homotopy, and in the case of $-1$, $\Psi^1 \circ H$ is the required homotopy.
            \end{proof}
        \subsection{A Chas-Sullivan type product}
            In this section we construct a product structure on the Thom spectrum $\sL^{-TL}$ (defined shortly). Under Hypothesis \ref{Looped Orientation Assumption}, the moduli spaces $\sM$ and $\sN$ (constructed in Section \ref{Other Moduli Subsection}) will have fundamental classes living in $R_*(\sL^{-TL})$. This structure will be crucial in proving Theorem \ref{Looped Theorem}.\par 
            We first note that by \cite{Milnor}, all spaces considered so far in this section are homotopy equivalent to CW complexes, and furthermore they are all Hausdorff and paracompact.\par 
            We denote the pullback of the virtual vector bundle $-TL \rightarrow L$ to $\sL$ also by $-TL$. We will define ring and module structures on the spectra $\sL^{-TL}$ and $\sL$, following \cite{Cohen}. These will be needed in Sections \ref{sec:5.6} and \ref{sec:5.7} to prove Theorem \ref{Looped Theorem}. \par 
            We have a commutative diagram:
            \[\begin{tikzcd}
                \sL_j \times_L \sL_k \arrow[r, "\Delta", hookrightarrow] \arrow[d, "c", hookrightarrow] & \sL_j \times \sL_k\\
                \sL_{j+k} & \\
            \end{tikzcd}\]
            where $\sL_j \mathop{\times}_L \sL_k$ is the fibre product of $\sL_j$ and $\sL_k$ over $L$, $\Delta$ is the natural inclusion map into $\sL_j \times \sL_k$, and $c$ is a concatenation map: given $(\gamma, \delta)$ in $\sL_j \mathop{\times}_L \sL_k$ and $t$ in $[0, 1]$, we define
            $$c(\gamma, \delta)(t) = 
            \begin{cases}
                \gamma(2t) & \mathrm{if}\, t \leq \frac{1}{2}\\
                \delta(2t - 1) & \mathrm{if}\, t \geq \frac{1}{2}\\
            \end{cases}$$
            Then $\sL_j \mathop{\times}_L \sL_k$ admits a tubular neighbourhood in $\sL_j \times \sL_k$ with normal bundle $TL$, and by applying  Definition \ref{PT without smoothness} with the virtual vector bundles $-TL \boxplus -TL$ and $-TL \boxplus 0$ over $\sL_j \times \sL_k$, we obtain two maps, both of which we denote by $\mu$:
            $$\mu: \sL_j^{-TL} \wedge \sL_k^{-TL} \xrightarrow{\Delta_\Shrek} \left(\sL_j \times_L \sL_k\right)^{-TL} \xrightarrow{c} \sL_{j + k}^{-TL}$$
            $$\mu: \sL_j^{-TL} \wedge \Sigma^\infty_+ \sL_k \xrightarrow{\Delta_\Shrek} \Sigma^\infty_+ \left(\sL_j \times_L \sL_k\right) \xrightarrow{c} \Sigma^\infty_+ \sL_{j + k}$$
            which fit together to give maps
            $$\mu: \sL^{-TL} \wedge \sL^{-TL} \rightarrow \sL^{-TL}$$
            $$\mu: \sL^{-TL} \wedge \Sigma^\infty_+\sL \rightarrow \Sigma^\infty_+\sL$$
            It follows from the coassociativity of $\Delta$ that these define a homotopy associative product on $\sL^{-TL}$, and a left module action of this on $\Sigma^\infty_+\sL$.\par 
            Let $i: L \hookrightarrow \sL$ be the inclusion of constant loops into $\sL_0$. The composition
            $$[L]: \Sphere \mathop{\rightarrow}^{[L]} L^{-TL} \mathop{\rightarrow}^{i} \sL^{-TL}$$
            defines a unit for $\sL^{-TL}$, similarly to \cite{Cohen}. Together, we have
            \begin{PP}
                $\sL^{-TL}$ is a ring spectrum, and $\Sigma^\infty_+\sL$ is a left module over it.\par 
                Furthermore if $R$ is a ring spectrum, $\sL^{-TL} \wedge R$ is a ring spectrum, and $\Sigma^\infty_+\sL \wedge R$ is a left module over it.
            \end{PP}
            It follows from construction that
            \begin{LL}\label{Product commutes with monodromy}
                All of the maps $\mu$ constructed in this subsection commute with the map $\Psi^1$ up to homotopy.
            \end{LL}
            \begin{RR}
                For $R$-homology classes $x$ and $y$ in the $R$-homology of $\sL^{-TL}$ or $\Sigma^\infty_+\sL$, we will often write $x \cdot y$ for $\mu(x, y)$ when it is unambiguous to do so.
            \end{RR}
            % Let $A$ and $B$ be closed manifolds of dimensions $a$ and $b$ respectively, along with maps $\alpha: A \rightarrow \sL$ and $\beta: B \rightarrow \sL$, such that the compositions $A, B \rightarrow L$ are transverse. Assume we are given $R$-orientations on $TA-TL$ and $TB-TL$.\par 
            % The fibre product $A \times_L B$ admits a natural map $A \times_L B \rightarrow \sL$, and its stable tangent bundle
            % $$T(A \times_L B) - TL \cong (TA-TL) + (TB-TL)$$
            % inherits a natural $R$-orientation.
            % \begin{LL}\label{prod is int}
            %     In this setting, the class $[A\times_L B]$ in $R_{a+b-n} \sL^{-TL}$ is given by
            %     $$[A \times_L B] = [A] \cdot [B]$$
            % \end{LL}
            % \begin{proof}
                
            % \end{proof}
        \subsection{The moduli spaces $\sM$, $\sN$ and $\sQ$}\label{Other Moduli Subsection}
            We fix $G$ a convex domain in $\C$ with smooth boundary, such that $\partial G$ contains $(-\eta, \eta)$ and $(-\eta, \eta) + i$ for some $\eta > 0$, and $G$ is symmetric in the lines $i\R$ and $\frac{i}{2} + \R$. We define 
            $$D_+ = (\R_{\leq 0} + i[0, 1]) \cup G \sqcup \{-\infty\}$$ 
            and 
            $$D_- = (\R_{\geq 0} + i[0, 1]) \cup G \sqcup \{+\infty\}$$
            viewed as the one-point compactifications of half-infinite strips. Both are compact Riemann surfaces with boundary, biholomorphic to a disc. 
            %We fix explicit identifications between $D_\pm$ and the standard disc $D^2$: 
            We fix biholomorphisms $\phi_\pm: D^2 \rightarrow D_\pm$ which send $\mp 1$ to $\mp \infty$, satisfying
            $$\phi_-(z) = i - \phi_+(-z)$$
            for all $z$, as shown below.\par 
            
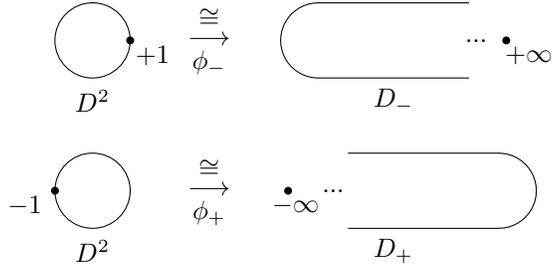
\vspace{0.5cm}
            {\centering\begin{tikzpicture}[line cap=round,line join=round,>=triangle 45,x=1.0cm,y=1.0cm]
    
    \draw(1,1) circle (0.5cm);
    
    \draw (2.55,1) node[anchor=center] {$\longrightarrow$};
    \draw (2.55, 1.3) node[anchor=center] {$\cong$};
    \draw (2.55, 0.7) node[anchor=center] {$\phi_-$};
    
    \draw (1,0.2) node[anchor=center] {$D^2$};
    
    \draw (1.8,0.8) node[anchor=center] {$+1$};
    
    \draw [shift={(4,1)}] plot[domain=1.57:4.71,variable=\t]({0.5 * cos(\t r)},{0.5 * sin(\t r)});
    
    \draw (4, 1.5)-- (6, 1.5);
    \draw (4, 0.5)-- (6, 0.5);
    
    \draw (6.1,1) node[anchor=center] {$...$};
    \draw (6.8,0.8) node[anchor=center] {$+\infty$};
    
    \draw (5, 0.2) node[anchor=center] {$D_-$};

    \draw(1, -1) circle (0.5cm);
    
    \draw (2.55, -1) node[anchor=center] {$\longrightarrow$};
    \draw (2.55, -0.7) node[anchor=center] {$\cong$};
    \draw (2.55, -1.3) node[anchor=center] {$\phi_+$};
    
    \draw (1, -1.8) node[anchor=center] {$D^2$};
    
    \draw (0.1, -1.2) node[anchor=center] {$-1$};
    
    \draw [shift={(6.4, -1)}] plot[domain=4.71:7.85, variable=\t]({0.5 * cos(\t r)}, {0.5 * sin(\t r)});
    
    \draw (4.4, -0.5)-- (6.4, -0.5);
    \draw (4.4, -1.5)-- (6.4, -1.5);
    
    \draw (4.2, -1) node[anchor=center] {$...$};
    \draw (3.7, -1.2) node[anchor=center] {$-\infty$};
    
    \draw (5, -1.8) node[anchor=center] {$D_+$};

    \begin{scriptsize}
        \fill [color=black] (1.5,1) circle (1.5pt);
        \fill [color=black] (6.5,1) circle (1.5pt);
        
        \fill [color=black] (0.5, -1) circle (1.5pt);
        \fill [color=black] (3.6, -1) circle (1.5pt);
    \end{scriptsize}
\end{tikzpicture}\par }
            \begin{DD}
                We define $\sR_\pm$ to be the space of smooth maps $w: D_\pm \rightarrow M$ such that for all $z \neq \mp \infty$ in $\partial D_\pm$, $w(z)$ lies in $\psi^{\im z}(L)$.
            \end{DD}
            \begin{RR}
                Since $D_\pm$ are compact, the topological energy $\int w^* \omega$ of any $w$ in $\sR_\pm$ is finite.
            \end{RR}
            Fix $J$ an $\omega$-tame almost complex structure on $M$ which is convex at infinity.
            
            % Let $\alpha, \beta: [0, 1] \rightarrow \partial D^2$ be smooth parametrisations of the boundary of the disc, injective except at their endpoints, such that $\alpha$ runs clockwise and sends 0 and 1 to 1 and $\beta$ runs anticlockwise and sends 0 and 1 to -1. \par 
            % \begin{center}
            %     \input{pic2.tikz}\par 
            % \end{center}
            % Fix $J$ an $\omega$-tame almost complex structure on $M$ which is convex at infinity. 
            % \begin{DD}
            %     We define $\sE$ to be the space of smooth maps $u: D^2 \rightarrow M$ such that for all $t$ in $[0, 1]$, $u(\alpha(t))$ lies in $\psi^t(L)$.\par 
            %     We define $\sF$ to be the space of smooth maps $u: D^2 \rightarrow M$ such that for all $t$ in $[0, 1]$, $u(\beta(t))$ lies in $\psi^t(L)$.
            % \end{DD}
            % There are natural evaluation maps $\sigma: \sE \rightarrow \sL_1$ and $\sigma: \sF \rightarrow \sL_{-1}$ sending $u$ to $t \mapsto u(\alpha(t))$ and $t \mapsto u(\beta(t))$ respectively. There are also natural evaluation maps $\pi: \sM \rightarrow L$ and $\pi: \sN \rightarrow L$ sending $u$ to $u(\alpha(0))$ and $u(\beta(0))$ respectively.\par 
            % There are evaluation maps $ev: (D^2, \partial D^2) \times \sE \rightarrow (M \times D^2, L_{\psi^1})$ and $ev: (D^2, \partial D^2) \times \sF \rightarrow (M \times D^2, L_{\psi^1})$, sending $(z, u)$ to $u(z)$ in both cases. Then there are bundle pairs $(ev^* TM, ev^* T^vL_{\psi^1})$ over $(D^2, \partial D^2) \times \sE$ and $(D^2, \partial D^2) \times \sF$. Note that these bundle pairs may have different Maslov indices over different connected components.
            \begin{DD}
                We define $\sM$ to be the moduli space of $J$-holomorphic maps in $\sR_-$, and $\sN$ to be the moduli space of $J$-holomorphic maps in $\sR_+$.\par 
                There are natural evaluation maps $\pi: \sR_\pm \rightarrow L$ sending $u$ to $u(\mp\infty)$. We define $\sQ$ to be the fibre product $\sN \times_L \sM$ with respect to these evaluation maps.
            \end{DD} 
            \begin{RR}
                Morally $\sQ$ is the limit as $l \rightarrow \infty$ of $\sP$.
            \end{RR}
            For generic $J$ the moduli spaces $\sM$ and $\sN$ are naturally smooth manifolds, and a similar argument to Section \ref{Moduli Section} shows that they are compact. Furthermore, for generic $J$ the two maps $\pi: \sM \rightarrow L$ and $\pi: \sN \rightarrow L$ are transverse, and hence $\sQ$ is also a compact smooth manifold. We assume we have chosen our $J$ sufficiently generically to satisfy all of this.\par 
            $\sM$ and $\sN$ may have different connected components which have different dimensions. 
            % We write 
            % $$\sM = \bigsqcup\limits_{i \geq 0} \sM^i$$
            % where $\sM^i$ is the component of $\sM$ lying in dimension $i$, and we write 
            % $$\sN = \bigsqcup\limits_{i \geq 0} \sN^i$$ 
            % similarly. 
            However, note that since both of these moduli spaces are compact, they only have finitely many non-empty components.
        \subsection{Bundle pairs on the moduli spaces}

            We define $\Sigma$ to be the space
            $$\left(D_- \cup [-1, +1] \cup D_+\right) / \sim$$
            where $+\infty \sim -1$ and $+1 \sim -\infty$, as shown below (identifying $D_\pm$ with discs, using $\phi_\pm$).
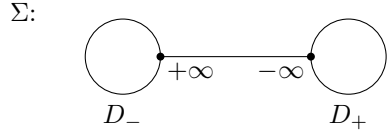
\par 
            \vspace{0.5cm}
            {\centering\begin{tikzpicture}[line cap=round,line join=round,>=triangle 45,x=1.0cm,y=1.0cm]
    \centering
    
    \draw (5,0.6) node[anchor=center]{\parbox{0 cm} {$\Sigma\mathchar\numexpr"6000+`:\relax$}};
    
    \draw(6.5,0) circle (0.5cm);
    \draw(9.5,0) circle (0.5cm);
    \draw (7,0)-- (9,0);
    \draw (6.5,-0.8) node[anchor=center] {$D_-$};
    \draw (9.5,-0.8) node[anchor=center] {$D_+$};
    \draw (7.4,-0.2) node[anchor=center] {$ +\infty $};
    \draw (8.6,-0.2) node[anchor=center] {$ - \infty $};
    \begin{scriptsize}
        \fill [color=black] (7,0) circle (1.5pt);
        \fill [color=black] (9,0) circle (1.5pt);
    \end{scriptsize}
\end{tikzpicture}\par }
            We define $\partial \Sigma \subseteq \Sigma$ to be the subspace 
            $$\left(\partial D_- \cup [-1, +1] \cup \partial D_+ \right) / \sim$$
            There are natural collapse maps $$\xi_\pm: (\Sigma, \partial \Sigma) \rightarrow (D_\pm, \partial D_\pm)$$
            which collapse $D_\mp$ and $[-1, 1]$ to $\mp \infty$, and are the identity on $D_\pm$.\par
            We define a submanifold $\tilde{L}_{\psi^1} \subseteq \C \times M$ to be the space of pairs $(z, x)$ such that $|z| = 1$ and if $z = 1$ then $x$ lies in $L$, otherwise $x$ lies in
            $$\psi^{\im \phi_-(z)} (L)$$
            This is diffeomorphic to $L_{\psi^1}$. We let $T^v \tilde{L}_{\psi^1}$ be its vertical tangent bundle.\par 
            There are natural maps $$ev_\pm: \left(D_\pm, \partial D_\pm\right) \times \sR_\pm \rightarrow \left(\C \times M, \tilde{L}_{\psi^1}\right)$$
            defined by
            $$ev_+(z, w) = \left(-\phi_+^{-1}(i + \bar{z}), w(z)\right)$$
            and
            $$ev_-(z, w) = \left(\phi_-^{-1}(z), w(z)\right)$$
            % $$ev_\pm(z, w) := \left(\mp \phi_\pm^{-1}(z), w(z)\right)$$
            We define a map
            $$C: (\Sigma, \partial \Sigma) \times \sQ \rightarrow \left(\C \times M, \tilde{L}_{\psi^1}\right)$$
            by
            $$C(z, (v, u)) := \begin{cases}
                    % \left(\phi_-^{-1}(z), u(z)\right) & \mathrm{if} \, z \in D_-\\
                    % \left(-\phi_+^{-1}(z), v(z)\right) & \mathrm{if} \, z \in D_+\\
                    \left(ev_-(z, u)\right) & \mathrm{if} \, z \in D_-\\
                    \left(ev_+(z, v)\right) & \mathrm{if} \, z \in D_+\\
                    \left(1, u(+\infty)\right) & \mathrm{if} \, z \in [-1, +1]
                \end{cases}
            $$
            We will consider the bundle pairs
            $$\xi_\pm^* ev_\pm^*\left(TM, T^v \tilde{L}_{\psi^1}\right),$$
            $$C^*\left(TM, T^v \tilde{L}_{\psi^1}\right)$$
            and 
            $$\xi_\pm^* \pi^*\left(TM, T^v\tilde{L}_{\psi^1}\right)$$
            From the construction, we see that the fibres of the complex parts of these bundle pairs are given as follows.
            $$\left(\xi^*_- ev_-^* TM\right)_{(z, (v, u))} = \begin{cases}
                T_{u(z)}M & \mathrm{if} \, z \in D_-\\
                T_{u(+\infty)}M & \mathrm{if} \, z \in D_+\\
                T_{u(+\infty)}M & \mathrm{if} \, z \in [-1, 1]
            \end{cases}$$
            $$\left(\xi_+^* ev_+^* TM\right)_{(z, (v, u))} = \begin{cases}
                T_{u(+\infty)} M & \mathrm{if} \, z \in D_-\\
                T_{v(z)} M & \mathrm{if} \, z \in D_+\\
                T_{u(+\infty)} M & \mathrm{if} \, z \in [-1, 1]
            \end{cases}$$
            $$\left(C^* TM\right)_{(z, (v, u))} = \begin{cases}
                T_{u(z)}M & \mathrm{if}\, z \in D_-\\
                T_{v(z)}M & \mathrm{if}\, z \in D_+\\
                T_{u(+\infty)}M & \mathrm{if}\, z \in [-1, 1]
            \end{cases}$$
            $$\left(\xi^*_\pm \pi^* TM\right)_{(z, (v, u))} = T_{u(+\infty)}M$$
            Similarly, the fibres of the real parts are given as follows, where $z$ now lies in $\partial \Sigma$.
            $$\left(\xi^*_- ev_-^* T^v\tilde{L}_{\psi^1}\right)_{(z, (v, u))} = \begin{cases}
                T_{u(z)}\psi^{\im z}(L) & \mathrm{if} \, z \in \partial D_-\\
                T_{u(+\infty)}L & \mathrm{if} \, z \in \partial D_+\\
                T_{u(+\infty)}L & \mathrm{if} \, z \in [-1, 1]
            \end{cases}$$
            $$\left(\xi_+^* ev_+^* T^v\tilde{L}_{\psi^1}\right)_{(z, (v, u))} = \begin{cases}
                T_{u(+\infty)} L & \mathrm{if} \, z \in \partial D_-\\
                T_{v(z)} \psi^{\im z}(L) & \mathrm{if} \, z \in \partial D_+\\
                T_{u(+\infty)} L & \mathrm{if} \, z \in [-1, 1]
            \end{cases}$$
            $$\left(C^* T^v\tilde{L}_{\psi^1}\right)_{(z, (v, u))} = \begin{cases}
                T_{u(z)}\psi^{\im z}(L) & \mathrm{if}\, z \in \partial D_-\\
                T_{v(z)}\psi^{\im z}(L) & \mathrm{if}\, z \in \partial D_+\\
                T_{u(+\infty)}L & \mathrm{if}\, z \in [-1, 1]
            \end{cases}$$
            $$\left(\xi^*_\pm \pi^* T^v\tilde{L}_{\psi^1}\right)_{(z, (v, u))} = T_{u(+\infty)}L$$
            % $$\begin{aligned}
            %     a(z, (v, u)) &{:= \begin{cases}
            %         \left(\phi_-^{-1}(z), u(z)\right) & \mathrm{if} \, z \in D_-\\
            %         \left(1, u(+\infty)\right) & \mathrm{otherwise}
            %     \end{cases}}\\
            %     b(z, (v, u)) &{:= \begin{cases}
            %         \left(-\phi_+^{-1}(z), v(z)\right) & \mathrm{if} \, z \in D_+\\
            %         \left(1, v(-\infty)\right) & \mathrm{otherwise}
            %     \end{cases}}\\
            %     c(z, (v, u)) &:= \begin{cases}
            %         \left(\phi_-^{-1}(z), u(z)\right) & \mathrm{if} \, z \in D_-\\
            %         \left(-\phi_+^{-1}(z), v(z)\right) & \mathrm{if} \, z \in D_+\\
            %         \left(1, u(+\infty)\right) & \mathrm{if} \, z \in [-1, +1]
            %     \end{cases}\\
            %     d(z, (v, u)) &:= \left(1, u(+\infty)\right)
            % \end{aligned}$$
            \begin{LL}\label{F}
                There is an isomorphism of bundle pairs $F$ from
                $$\xi_-^* ev_-^*\left(TM, T^v \tilde{L}_{\psi^1}\right) \oplus \xi_+^* ev_+^*\left(TM, T^v \tilde{L}_{\psi^1}\right)$$
                to
                $$C^* \left(TM, T^v \tilde{L}_{\psi^1}\right) \oplus \xi_+^* \pi^* \left(TM, T^v \tilde{L}_{\psi^1}\right)$$
            \end{LL}
            \begin{proof}
                We define $F_{(z, (v, u))}$ explicitly as follows. \par 
                If $z$ lies in $D_-$, then 
                $$F_{(z, (v, u))}: T_{u(z)}M \oplus T_{u(+\infty)}M \rightarrow T_{u(z)} M \oplus T_{u(+\infty)} M$$
                is given by the identity.\par 
                If $z$ lies in $D_+$, then
                $$F_{(z, (v, u))}: T_{u(+\infty)} M \oplus T_{v(z)} M \rightarrow T_{v(z)} M \oplus T_{u(+\infty)} M$$
                is given by the matrix
                $$\begin{pmatrix}
                    0 & 1\\
                    -1 & 0
                \end{pmatrix}$$
                with respect to the direct sum decomposition above.\par 
                If $z$ lies in $[-1, +1]$, then
                $$F_{(z, (v, u))}: T_{u(+\infty)}M \oplus T_{u(+\infty)}M \rightarrow T_{u(+\infty)}M \oplus T_{u(+\infty)}M$$
                is given by the matrix
                $$\begin{pmatrix}
                    \cos \left(\frac{z + 1}{4} \pi\right) & \sin \left( \frac{z + 1}{4} \pi \right) \\
                    - \sin \left(\frac{z + 1}{4} \pi\right) & \cos \left( \frac{z + 1}{4} \pi \right)
                \end{pmatrix}$$
                with respect to the decomposition above.\par 
                Note that this map $F$ respects the totally real subbundles when $z$ lies in $\partial D$. 
            \end{proof}
            Note that 
            $$\Ind \, \pi^* \left(TM, T^v \tilde{L}_{\psi^1}\right)$$
            is isomorphic to $\pi^* TL$, over both $\sR_\pm$.\par
            Now pick a ``pinch'' map $r: \left(D^2, \partial D^2\right) \rightarrow (\Sigma, \partial \Sigma)$ such that the compositions $\xi_\pm \circ r$ induce maps of degree 1 from $\partial D^2$ to $\partial D_\pm$. This description determines $r$ up to homotopy, but for convenience we will make a more specific choice of $r$ later. 
            \begin{CC}\label{rF}
                There is an isomorphism of bundle pairs $r^*F$ from
                $$r^*\xi_-^* ev_-^*\left(TM, T^v \tilde{L}_{\psi^1}\right) \oplus r^*\xi_+^* ev_+^*\left(TM, T^v \tilde{L}_{\psi^1}\right)$$
                to
                $$r^*C^* \left(TM, T^v \tilde{L}_{\psi^1}\right) \oplus r^*\xi_+^* \pi^* \left(TM, T^v \tilde{L}_{\psi^1}\right)$$
            \end{CC}

        \subsection{Gluing and orientations of the moduli spaces}
            % Let $\sU_\infty = \sQ$. Then as in Section \ref{Moduli Section}, for a generic ($l$-dependent) perturbation of the almost complex structure, the space
            % $$\sV' := \bigsqcup\limits_{0 \leq l < \infty} \sU_l$$
            % is a smooth manifold such that the natural map $\sV' \rightarrow \R$ is proper, and it follows from the proof of gluing in \cite{EkholmIvan} that
            % $$\sV := \bigsqcup\limits_{0 \leq l \leq \infty} \sU_l$$
            % is naturally a compact smooth manifold with boundary $L \sqcup \sQ$. 
            The proof of gluing in \cite{EkholmIvan} shows that there is a diffeomorphism $\sQ \cong \sP_{l+1}$ for sufficiently large $l$, so $\sW$ can be viewed as a bordism from $L$ to $\sQ$. Here we take $\sP_{l+1}$ and $\sW$ as in Section \ref{Moduli Section}.
            \cite{EkholmIvan} uses in an important way the fact that there is no disc bubbling (which follows from relative exactness of $L$).\par  
            The vertical tangent bundle of this cobordism restricted to $\sQ$ is then naturally isomorphic to
            $$\Ind\, r^* C^* \left(TM, T^v \tilde{L}_{\psi^1}\right)$$
            \begin{RR}
                $\Ind\, r^* C^* \left(TM, T^v \tilde{L}_{\psi^1}\right)$ is isomorphic to the index bundle $\Ind$ constructed in Section \ref{Orienty Subsection}. For the purposes of brevity we will therefore refer to it as $\Ind$ for the rest of Section \ref{Looped Section}.
            \end{RR}
            Therefore an $R$-orientation of $\Ind$ induces an $R$-orientation of $T\sQ$.\par  
            By construction, $T\sM$ and $T\sN$ are naturally stably isomorphic to
            $$\Ind\,r^*\xi_\mp^*ev_\mp^* \left(TM, T^v \tilde{L}_{\psi^1}\right)$$
            respectively. Therefore $R$-orientations on both of these as well as $TL$ induce one on $T\sQ$, since $\sQ = \sN \times_L \sM$ implies that
            $$T\sQ = T\sN + T\sM - \pi^* TL$$
            \begin{BB}\label{Looped Orientation Assumption}
                There are $R$-orientations of $T\sN$, $T\sM$, $\Ind$ and $L$, such that the two induced $R$-orientations on $T\sQ$ above agree under the isomorphism $r^*F$ constructed in Lemma \ref{F}. Furthermore the restriction of the $R$-orientation of $\Ind$ to the space of constant maps $L$ is the given $R$-orientation of $L$.
            \end{BB}
            \begin{LL}\label{Gluing Lemma}
                Under Assumption \ref{Looped Orientation Assumption}, the cobordism $\sW$ admits an $R$-orientation, restricting to the given ones on $L$ and $\sQ$.
            \end{LL}
            \begin{proof}
                The cobordism $\sW$ has tangent bundle stably isomorphic to $\Ind \oplus \R$, and therefore the $R$-orientation on $\Ind$ induces an $R$-orientation on $\sW$. By the last part of Assumption \ref{Looped Orientation Assumption}, this $R$-orientation restricts on the space of constant maps $L$ to the given $R$-orientation of $L$.\par 
                $\sQ$ now admits two $R$-orientations: one coming from the $R$-orientation on $\sW$, which has $\sQ$ as a boundary component (by identifying $\sP_{l+1}$ with $\sQ$ via the Ekholm-Smith diffeomorphism), and the one coming from the $R$-orientations on $T\sM$, $T\sN$ and $TL$ under the identification 
                $$T\sQ = T\sN + T\sM - \pi^* TL$$
                It will therefore suffice to check that the natural bundle isomorphism which comes from gluing
                $$\Ind \cong T\sQ \cong T\sN + T\sM - \pi^* TL $$
                from above agrees with the one induced by the isomorphism $F$ from Lemma \ref{F}, up to homotopy of bundle isomorphisms. Assumption \ref{Looped Orientation Assumption} then tells us that the $R$-orientations on both sides agree. To carry this out we need to open up both of these isomorphisms.\par 
                We will recall a sketch of how Ekholm and Smith, in \cite{EkholmIvan}, construct a diffeomorphism from $\sQ$ to $\sP_{l + 1}$ (where $\sP_{l + 1}$ is defined as in Section \ref{Moduli Section}). Fix some very large $l > 0$. We define $D_{\pm, l}$ to be the subset of $D_\pm$ given by
                $$D_\pm \cap \left\{\pm \re \geq -(l - 1)\right\}$$
                which we view as subsets of $G_{l} = C_{l + 1}$ ($G_l$, $C_{l + 1}$ defined as in Section \ref{Moduli Section}) by translating by $\pm l$. We let $H$ be the closure of the complement in $G_l$ of the union of these two regions:
                $$H := [-1, 1] + i[0, 1]$$
                Ekholm and Smith construct a pre-gluing embedding 
                $$PG: \sQ \hookrightarrow \sD_{f_{l + 1}, C_{l + 1}}$$
                sending $(v, u)$ to a map $(u \#_c v): C_{l + 1} \rightarrow M$ which agrees with $u$ and $v$ on $D_{\pm, l}$ respectively, and is $C^1$-small when restricted to $H$. Roughly, this is constructed by picking a metric on $M$, picking $l$ large enough so that outside $D_{\pm, l}$ the images of both $u$ and $v$ lie inside a single geodesically convex neighbourhood of $M$, and cutting them off with a bump function.\par 
                We view $T\sQ$ as lying inside 
                $$\Gamma\left(r^* \xi_+^* ev_+^*\left(TM, T^v\tilde{L}_{\psi^1}\right) \oplus r^* \xi_-^* ev_-^*\left(TM, T^v\tilde{L}_{\psi^1})\right)\right) $$
                and $T\sP_{l + 1}$ as lying inside 
                $$\Gamma\left(r^* C^*\left(TM, T^v\tilde{L}_{\psi^1}\right)\right)$$
                Note that we can do this without introducing stabilisations because our moduli spaces are cut out transversely, and so our choices of almost complex structure induce surjective Cauchy-Riemann operators.\par 
                Using a Newton-Picard iteration, Ekholm and Smith show that there is a diffeomorphism $\rho: PG(\sQ) \rightarrow \sP_{l + 1}$, such that $PG$ and $\rho \circ PG$ are $C^1$- close.
                Therefore the derivative $d (\rho \circ PG)$ sends a pair of sections $(s_+, s_-)$ lying in $T\sQ$, to a section $s$ lying in $T\sP_{l + 1}$, such that $s$ is $C^0$-small on $H$ and $C^0$-close to $u$ and $v$ when restricted to $D_{\pm, l}$.\par 
                We make some choice of ``pinch'' map $r: C_{l + 1} \rightarrow \Sigma$ which sends $\partial C_{l + 1}$ to $\partial \Sigma$, and which restricts to the identity on $D_{\pm, l}$. Note that $ev \simeq C \circ r$ on $C_{l+1}$. \par 
                By construction, we see that the map induced by $r^*F$ on sections of these bundle pairs sends a pair of sections $(s_\pm)$, whose evaluations at $\mp \infty$ agree, to a section $s$, whose restrictions to $D_{\pm, l}$ agree with $s_\pm$ respectively. Consider the composition $A:$
                $$\begin{tikzcd}
                    T\sQ \arrow[d, hookrightarrow]\\
                    \Gamma\left(r^*\xi_-^*ev_-^*\left(TM, T^v \tilde{L}_{\psi^1}\right) \oplus r^* \xi_+^* ev_+^*\left(TM, T^v \tilde{L}_{\psi^1}\right)\right) \arrow[d, "r^*F"]\\
                    \Gamma\left(r^*C^*\left(TM, T^v \tilde{L}_{\psi^1}\right) \oplus r^* \xi_+^* \pi^*\left(TM, T^v \tilde{L}_{\psi^1}\right)\right) \arrow[d, "\mathrm{Projection\, to\, the\, first\, factor}"]\\
                    \Gamma\left(r^*C^*\left(TM, T^v \tilde{L}_{\psi^1}\right)\right) \arrow[d, "\mathrm{Orthogonal\, projection}"]\\
                    T\sP_{l + 1}\\
                \end{tikzcd}$$
                where we weight the metric on the bundle pair to be small on $H$. Note that since $r^*F$ respects the direct sum decomposition away from $H$, if the weight is sufficiently small on $H$, any Cauchy-Riemann operator on the bundle pair in either the second or third term which respects the direct sum decomposition is of distance less than 1 to a Cauchy-Riemann operator on the other which respects the direct sum. Therefore the map $T\sQ \rightarrow T\sP_{l + 1}$ coming from Lemma \ref{Index Bundles Are Canonical} agrees up to homotopy with $A$, and so under Assumption \ref{Looped Orientation Assumption}, $A$ respects the $R$- orientations. \par 
                $A$ is $C^0$-close to $d(\rho \circ PG)$ and so the two maps are homotopic isomorphisms of vector bundles, which is what we wanted.
            \end{proof}
            \begin{LL}\label{Corientations with Looooooops}
                If $T^v \tilde{L}_{\psi^1}$ admits a stable trivialisation over an $(i + 1)$-skeleton of $\tilde{L}_{\psi^1}$, then the induced stable trivialisations over an $i$-skeleton $\sQ_i$ of $\sQ$ of the bundle pairs appearing in Corollary \ref{rF} agree up to homotopy, under $r^*F$.
            \end{LL}
            \begin{proof}
                Homotoping our maps if necessary, we can assume that our maps 
                $$\left(D^2, \partial D^2\right) \times \sQ \rightarrow \left(\C \times M, \tilde{L}_{\psi^1}\right)$$
                send $\partial D^2 \times \sQ_i$ to an $(i+1)$-skeleton of $\tilde{L}_{\psi^1}$.\par 
                Then the two stable trivialisations are related by a map $\eta: \partial D^2 \times \sQ_i \rightarrow O$, where $O$ is the infinite orthogonal group. We will use the choice of $r$ from the proof of Lemma \ref{Gluing Lemma} for convenience, noting that this is unique up to homotopy.\par 
                From the construction of $F$, for any $y$ in $\sQ$ and $x$ in $r|_{\partial D^2}^{-1} \partial D_-$, $\eta(x, y)$ is given by the identity matrix. For $x$ in $r|_{\partial D^2}^{-1} \partial D_+$, $\eta(x, y)$ is given by (the stabilisation of) the block matrix
                $$\begin{pmatrix}
                    0 & 1\\
                    -1 & 0
                \end{pmatrix}$$
                Then along the top part of $\partial D^2$, $\eta(\cdot, y)$ is given by some homotopy between these two matrices, and travelling in the opposite direction along the bottom produces the reverse of this homotopy. Therefore for a fixed $y$, $\eta(\cdot, y)$ is a contractible loop. However these loops are independent of $y$, so the contractions can be chosen to be independent of $y$ too, and so $\eta$ is nullhomotopic.
            \end{proof}
            Then from Proposition \ref{Complex Orientability} along with the fact that $\pi_3 BO = 0$, it follows that%%%
            \begin{CC}\label{vc}
                \begin{enumerate}
                    \item If $T^v L_{\psi^1}$ admits a stable trivialisation over a 2-skeleton of $L_{\psi^1}$, then Assumption \ref{Looped Orientation Assumption} holds for $R = H\Z$.
                    \item If $T^v L_{\psi^1}$ admits a stable trivialisation over a 2-skeleton of $L_{\psi^1}$, then Assumption \ref{Looped Orientation Assumption} holds for $R = KU$. %%%
                    % \item If $T^v L_{\psi^1}$ admits a stable trivialisation over a 3-skeleton of $L_{\psi^1}$, then Assumption \ref{Looped Orientation Assumption} holds for $R = KU$.
                    \item If $T^v L_{\psi^1}$ admits a stable trivialisation, then Assumption \ref{Looped Orientation Assumption} holds when $R$ is any complex-oriented cohomology theory.
                \end{enumerate}
            \end{CC}
            \begin{proof}
                Assume $T^v\tilde{L}_{\psi^1}$ admits a stable trivialisation over an $(i+1)$-skeleton. Note that since $\pi_3 BO = 0$, if this holds for $i=2$, it holds for $i=3$ too.\par 
                We use that $T\sQ$ is naturally stably isomorphic to $\Ind$, and $T\sM$ and $T\sN$ are naturally stably isomorphic to
                $$\Ind\,r^*\xi^*_\mp ev^*_\mp\left(TM, T^v \tilde{L}_{\psi^1}\right)$$
                respectively. Since these all are index bundles of bundle pairs pulled back via some map to $(M \times \C, \tilde{L}_{\psi^1})$, these maps can be homotoped to send an $i$-skeleton of the moduli space times $\partial D^2$ to an $(i+1)$-skeleton of $\tilde{L}_{\psi^1}$. Therefore this stable trivialisation induces a stable complex structure on each an $i$-skeleton of these moduli spaces, by Proposition \ref{Complex Orientability}. We then use the fact that a complex structure over a 1-skeleton induces an orientation and a complex structure over a 2-skeleton induces a $KU$-orientation.\par 
                We then apply Proposition \ref{Corientations with Looooooops} to see that these induced stable complex structures agree under the isomorphism $r^*F$.\par 
                Finally, use the fact that a stable complex structure over a 1-skeleton induces an orientation, and one over a 2-skeleton induces a $KU$-orientation.
            \end{proof}
            \begin{proof}[Proof of Proposition \ref{MN Orientations}]
                We will show that if $TL$ admits a homotopy class of stable trivialisations over an $i$-skeleton $L_i$ of $L$ which is preserved by $\psi^1$ and can be extended to an $(i+1)$-skeleton $L_{i + 1}$, then $T^v L_{\psi^1}$ admits a stable trivialisation over an $i$-skeleton of $L_{\psi^1}$. Then Corollary \ref{vc} will imply the result.\par 
                Let $\tilde{\psi^1}:L \rightarrow L$ be a map homotopic to $\psi^1|_L$ which sends $L_i$ to itself. Then $L_{\psi^1}$ is homotopy equivalent to 
                $$L \times [0, 1] / (0, x) \sim (1, \tilde{\psi^1}(x))$$
                and an $(i+1)$-skeleton of this is given by $$\left(L_i \times [0, 1]\right) \cup L_{i+1}$$
                Then by assumption, $T^v L_{\psi^1}$ admits a stable trivialisation over this $(i+1)$-skeleton.
            \end{proof}
        \subsection{Composition of the moduli spaces}\label{sec:5.6}
            We now assume Assumption \ref{Looped Orientation Assumption} holds throughout the rest of Section \ref{Looped Section}.\par 
            Fix some diffeomorphism $\tilde{L}_{\psi^1} \cong L_{\psi^1}$ covering some orientation-reversing diffeomorphism $\partial D^2 \cong S^1$ sending $1 \in \partial D^2$ to $0 \in S^1 = \R/\Z$. There are natural evaluation maps 
            $$S^1 \times \sR_\pm \xrightarrow{\exp(-2\pi i \cdot)} \partial D^2 \times \sR_\pm \xrightarrow{\phi_\pm(\mp \cdot)} \partial D_\pm \times \sR_\pm \xrightarrow{ev_\pm} \tilde{L}_{\psi^1} \xrightarrow{\cong} L_{\psi^1}$$
            whose adjoints define maps
            $$\sigma_\pm: \sR_\pm \rightarrow \sL_{\mp 1}$$
            % Then we let $\sigma_+ = \sigma'_+$ and $\sigma_- = \bar{\sigma}'_-$, where $\bar{\cdot}$ denotes the reverse of a loop. 
            % These are maps
            % $$\sigma_\pm: \sR_\pm \rightarrow \sL_{\pm 1}$$
            Therefore $\sQ$ admits a natural evaluation map 
            $$\sigma: \sQ \rightarrow \sL_0$$
            sending $(v, u)$ to the concatenation $c(\sigma_-(v), \sigma_+(u))$. Then choices of $R$-orientations of $L$ and all the moduli spaces allow us to use $\sigma_\pm$ and $\sigma$ to define fundamental classes $[\sM]$, $[\sN]$ and $[\sQ]$ in $\bigoplus\limits_j R_j\left(\sL_{\pm 1}^{-TL}\right)$ and $\bigoplus\limits_j R_j\left(\sL_0^{-TL}\right)$ respectively, as in Section \ref{Fundamental Classes Subsection}.\par 
            
            % $$\sigma: \sE \rightarrow \sL_1$$
            % and 
            % $$\sigma: \sF \rightarrow \sL_{-1}$$ sending $u$ to $u|_{\partial D_-}$ and $\overline{u|_{\partial D_+}}$ respectively, where we pick smooth clockwise parametrisations of the boundaries of $D_\pm$ based at $\pm \infty$, and $\overline{\cdot}$ denotes the reverse of a loop. Therefore 
            % $\sQ$ admits a natural map $\sigma: \sQ \rightarrow \sL_0$, sending $(\gamma, \delta)$ to their concatenation $c(\gamma, \delta)$. 
            Our goal in this subsection is to prove the following two lemmas.
            \begin{LL}\label{Product Is Intersection}
                $$[\sN] \cdot [\sM] = \left[\sQ\right]$$
                in $\bigoplus\limits_{j} R_j\left(\sL_0^{-TL}\right)$.
            \end{LL}
            \begin{LL}\label{Gluing}
                $$\left[\sQ\right] = [L]$$ 
                in $\bigoplus\limits_{j} R_j\left(\sL_0^{-TL}\right)$.
            \end{LL}
            From these, we deduce 
            \begin{LL} \label{Composition Is A Unit}
                The composition\\
                \begin{tikzcd}
                    \Sigma^\infty_+\sL_0 \wedge R \arrow[r, "{[\sM]\cdot}"] &
                    \bigvee\limits_j \Sigma^{\infty + j}_+ \sL_1 \wedge R \arrow[r, "{[\sN]\cdot}"] &
                    \bigvee\limits_j \Sigma^{\infty + j}_+ \sL_0 \wedge R \arrow[r, "{p}"] &
                    \Sigma^\infty_+ \sL_0 \wedge R
                \end{tikzcd}\\
                is an equivalence, where $p: \bigvee\limits_j \Sigma^{\infty + j}_+ \sL_0 \rightarrow \Sigma^\infty_+ \sL_0$ is the natural projection map.
            \end{LL}
            \begin{proof}[Proof of Lemma \ref{Product Is Intersection}]
                Consider the following diagram.
                \[\begin{tikzcd}
                    \Sphere \arrow[d, "{[\sN] \wedge [\sM]}"] \arrow[dr, "{\left[\sQ\right]}"]& \\
                    \sN^{-T\sN} \wedge R \wedge \sM^{-T\sM} \wedge R \arrow[d, "\mathrm{Thom}"] \arrow[r, "i_\Shrek"] & 
                    \sQ^{-\left(T\sN + T\sM - \pi^*TL\right)} \wedge R \arrow[d, "\mathrm{Thom}"]\\
                    \bigvee\limits_j \Sigma^j \sN^{-\pi^*TL} \wedge R \wedge \sM^{-\pi^*TL} \wedge R \arrow[d, "\sigma_- \wedge \sigma_+"] \arrow[r, "i'_\Shrek"]&
                    \bigvee\limits_j \Sigma^j \sQ^{-\pi^*TL} \wedge R \arrow[d, "\sigma"] \\
                    \bigvee\limits_j \Sigma^j \sL_{-1}^{-TL} \wedge R \wedge \sL_{+1}^{-TL} \wedge R \arrow[r, "\mu"]&
                    \bigvee\limits_j \Sigma^j \sL_0^{-TL} \wedge R
                \end{tikzcd}\]
                The two arrows labelled $\mathrm{Thom}$ are the isomorphisms from the Thom isomorphism theorem for our choices of $R$-orientations on $T\sM-\pi^* TL$ and $T\sN - \pi^* TL$ followed by inclusion into this wedge product, and $i_!$ and $i'_!$ are obtained by applying Definition \ref{PT without smoothness} to the embedding $$\sQ \hookrightarrow \sN \times \sM$$
                using the virtual vector bundles $-T\sN \boxplus -T\sM$ and $-\pi^*TL \boxplus -\pi^*TL$ respectively.\par 
                % The top triangle commutes by Lemma \ref{prod is int} applied to $\sM$ and $\sN$.\par 
                The top triangle commutes by Lemma \ref{PT functorial}, the middle square commutes by naturality of the Thom isomorphism, and the bottom square commutes by construction of the product map $\mu$. Therefore the entire diagram commutes.\par 
                The composition down along the left and across is given by:
                $$\mu \circ (\sigma_- \wedge \sigma_+) \circ \mathrm{Thom} \circ [\sN] \wedge [\sM] = [\sN] \cdot [\sM]$$
                whereas composition down along the right is given by
                $$\sigma \circ \mathrm{Thom} \circ [\sQ] = [\sQ]$$
                so the result follows.
            \end{proof}
            \begin{proof}[Proof of Lemma \ref{Gluing}]
                By Lemma \ref{Gluing Lemma}, there exists an $R$-orientable cobordism $\sW$ from $\sQ$ to $L$, with respect to the given $R$-orientations on both ends. Furthermore, evaluating along the boundary allows us to extend the maps $\sigma: \sQ \rightarrow \sL_0$ and $L \hookrightarrow \sL_0$ to the entirety of $\sW$. Therefore the result follows from Lemma \ref{Cobodism Implies Same Fundamental Class}.
            \end{proof}
        \subsection{Proof of Theorem \ref{Looped Theorem}}\label{sec:5.7}
            \begin{proof}[Proof of Theorem \ref{Looped Theorem}]
                Lemma \ref{Product commutes with monodromy} implies that the following diagram commutes up to homotopy:
                \[\begin{tikzcd}
                    \Sigma^\infty_+\sL_0 \wedge R \arrow[r, "{[\sM] \cdot}"] \arrow[d, "\Psi^1"] & 
                    \bigvee\limits_{j} \Sigma^{\infty + j}_+\sL_1 \wedge R \arrow[r, "{[\sN] \cdot}"] \arrow[d, "\Psi^1"] & 
                    \bigvee\limits_{j} \Sigma^{\infty + j}_+\sL_0 \wedge R \arrow[d, "\Psi^1"]\arrow[r, "p"] & 
                    \Sigma^\infty_+\sL_0 \wedge R \arrow[d, "\Psi^1"]\\ 
                    \Sigma^\infty_+\sL_0 \wedge R \arrow[r, "{\Psi^1_*[\sM] \cdot}"] & 
                    \bigvee\limits_{j} \Sigma^{\infty + j}_+\sL_1 \wedge R \arrow[r, "\Psi^1_*{[\sN]}\cdot"] & 
                    \bigvee\limits_{j} \Sigma^{\infty + j}_+\sL_0 \wedge R \arrow[r, "p"] & 
                    \Sigma^\infty_+\sL_0 \wedge R
                \end{tikzcd}\] 
                By Lemma \ref{Monodromy is trivial on L1}, the second vertical arrow is homotopic to the identity, and the horizontal arrows along the top are all homotopic to the corresponding horizontal arrows along the bottom. Furthermore Lemma \ref{Composition Is A Unit} tells us that the composition along the top (and similarly along the bottom) is an equivalence. It follows that the map 
                $$\Psi^1: \Sigma^\infty_+ \sL_0 \wedge R \rightarrow \Sigma^\infty_+ \sL_0 \wedge R$$ 
                is homotopic to the identity. Note that all maps and homotopies we used here were $R$-linear.
            \end{proof}
            \section{Families over other bases}\label{Other Bases Section}
        Our goal in this section will be to prove Theorem \ref{Bigger Bases}. In this section, we only consider mod-2 singular homology. We assume all spaces in this section are homotopy equivalent to CW complexes. We will use the following pair of purely topological lemmas:
        \begin{LL}[{\cite[Lemma 4.3]{Lalonde-McDuff}}]\label{Bootstrapping Upwards}
            Fix fibre bundles $A \hookrightarrow B \twoheadrightarrow C$ and $F \hookrightarrow E \twoheadrightarrow B$. Assume that the fibre bundles
            $$E|_A \hookrightarrow E \twoheadrightarrow C$$
            and
            $$F \hookrightarrow E|_A \twoheadrightarrow A$$
            % \begin{enumerate}
            %     \item $A \hookrightarrow B \twoheadrightarrow C$
            %     \item $E|_A \hookrightarrow E \twoheadrightarrow C$
            %     \item $F \hookrightarrow E|_A \twoheadrightarrow A$
            % \end{enumerate}
            both c-split. Then the fibre bundle $F \hookrightarrow E \twoheadrightarrow B$ c-splits.
        \end{LL}
        \begin{proof}
            The inclusion of a fibre $F \hookrightarrow E$ is the composition $F \hookrightarrow E|_A$ and $E|_A \hookrightarrow E$. The hypothesis of the lemma is that both these maps induce injections on mod-2 singular homology.
        \end{proof}
        \begin{LL}[{\cite[Lemma 4.1(ii)]{Lalonde-McDuff}}]\label{Surjection Lemma}
            Let $E \twoheadrightarrow B$ be a fibre bundle, and ${f: B' \rightarrow B}$ a map which is surjective on mod-2 singular homology. Assume that the pullback bundle $f^* E \twoheadrightarrow B'$ c-splits. Then $E \twoheadrightarrow B$ c-splits.
        \end{LL}
        \begin{proof}[Proof of Theorem \ref{Bigger Bases}]
            We first assume that $X = (S^1)^i$.\par 
            The case $i = 0$ is trivial and the case $i = 1$ follows from Theorem \ref{HLL theorem} along with an application of the Mayer-Vietoris sequence. We proceed by induction, and assume we know the result for $i - 1$, for all $M$ and $L$ as in the statement of Theorem \ref{HLL theorem}.\par
            We let $E = \gamma^* \sE$. Then we have two fibre bundles
            $$L \hookrightarrow E \twoheadrightarrow (S^1)^i$$
            and
            $$S^1 \hookrightarrow (S^1)^i \twoheadrightarrow (S^1)^{i-1}$$
            where the inclusion map is inclusion to the first factor and the projection map is projection to all the other factors.\par 
            The fibre bundle $E|_{S^1} \twoheadrightarrow S^1$ c-splits by Theorem \ref{HLL theorem}. So by Lemma \ref{Bootstrapping Upwards}, it suffices to show the fibre bundle given by the composition $$E \twoheadrightarrow (S^1)^i \twoheadrightarrow (S^1)^{i-1}$$ c-splits.\par 
            For $\tau$ in $(S^1)^{i-1}$, we let $\gamma_\tau$ be the map $S^1 \rightarrow Lag_L$ given by $\gamma_\tau(t) = \gamma(t, \tau)$.\par 
            We perturb the map $\gamma: S^1 \times (S^1)^{i-1}$ so that for each $\tau$ in $(S^1)^{i-1}$, $\gamma_\tau(t)$ is constant in $t$ in a neighbourhood of 0.
            \begin{MM}\label{Existence of Hams}
                There exists a smooth map $H: S^1 \times (S^1)^{i-1} \rightarrow \R$ such that for $\tau$ in $(S^1)^{i-1}$, the Hamiltonian $H_\tau: M \times S^1 \rightarrow \R$ generates the Hamiltonian isotopy of Lagrangians $\gamma_\tau$.\par 
                More explicitly, this means that the Hamiltonian flow of $H_\tau$ applied to the Lagrangian $\gamma_\tau(0)$ is $\gamma_\tau$.
                % There exists a smooth family of smooth maps $H_{\tau}: M \times S^1 \rightarrow \R$ for $\tau$ in $(S^1)^{i-1}$ such that the Hamiltonian flow of $H_\tau$ applied to $\gamma_\tau(0)$ is $\gamma_\tau$.
            \end{MM}
            \begin{proof}[Proof of Claim \ref{Existence of Hams}]
                For each $(t_0, \tau)$  in $S^1 \times (S^1)^{i-1}$, let $V_{t_0, \tau}$ be the space of compactly supported smooth maps $f: M \rightarrow \R$ such that the Hamiltonian flow of $f$ applied to $\gamma_\tau(t_0)$ agrees with $\gamma_\tau(t_0 + \cdot)$ to first order.\par 
                Together these determine a fibre bundle over $S^1 \times (S^1)^{i-1}$. Each $V_{t_0, \tau}$ is convex and hence contractible, so we can choose some smooth section of this fibre bundle. This gives us the map $H \times S^1 \times (S^1)^{i-1} \rightarrow \R$ that we require.
                % For each point $(t_0, \tau)$ in $S^1 \times (S^1)^{i-1}$, let $V_{t_0, \tau}$ be the space of germs near $M \times \{0\}$ of compactly supported smooth maps $M \times (-\varepsilon, \varepsilon) \rightarrow \R$ for some $\varepsilon > 0$, whose Hamiltonian flow $\phi_t$ applied to $\gamma(t_0, \tau)$ agrees with $\gamma_\tau(t_0 + t)$, for $t$ in $(-\varepsilon, \varepsilon)$. Each $V_{t_0, \tau}$ is convex and hence contractible. Together these define a fibre bundle over $(S^1)^i$, and a smooth section of this bundle determines the family $H_\tau$ that we want.
                % %and the family of $H_\tau$ we want is a choice of smooth section of this bundle. 
                % The fibres of this bundle are contractible so this exists.
            \end{proof}
            We now use this family of Hamiltonians $H_\tau$ to construct a suspension of each Lagrangian isotopy $\gamma_\tau$, as follows.  We define a map
            $$f: E \hookrightarrow M \times T^* S^1 \times (S^1)^{i-1}$$
            by 
            $$f(x) = (x, t, H_\tau(x, t), \tau)$$
            where $x$ lies in the fibre over $(t, \tau)$ in $S^1 \times (S^1)^{i-1}$, and we identify $T^* S^1 $ with $S^1 \times \R$ in the usual way. Note that $E \rightarrow (S^1)^{i-1}$ is a fibre bundle.\par 
            A direct computation shows that for each $\tau$, this gives a relatively exact Lagrangian in $M \times T^* S^1$, therefore this realises $E \twoheadrightarrow (S^1)^{i-1}$ as a family of relatively exact Lagrangians in $M \times T^*S^1$, and therefore by the induction hypothesis it c-splits. \par 
            Therefore by Lemma \ref{Bootstrapping Upwards}, the fibre bundle $E \twoheadrightarrow (S^1)^{i}$ c-splits.\par 
            For the general case, we consider the composition $(S^1)^i \xrightarrow{f} X \xrightarrow{\gamma} Lag_L$. Since $(f \circ \gamma)^* \sE \twoheadrightarrow (S^1)^i$ c-splits, by Lemma \ref{Surjection Lemma}, $\gamma^* \sE \twoheadrightarrow X$ also c-splits.
        \end{proof}
        % \begin{proof}[Proof of Theorem \ref{Bigger Bases}]
        %     Follows from Lemma \ref{Surjection Lemma} by considering a degree 1 map $(S^1)^i \rightarrow X$.
        % \end{proof}
        \appendix
            \section{An example in complex $K$-Theory}\label{K-theory Example}
                Our goal in this section is to prove the following.
                \begin{PP}\label{Kexample}
                    In all sufficiently high dimensions $n$, there exists a closed $n$-dimensional manifold $V$ such that:
                    \begin{enumerate}
                        \item $V$ is stably parallelisable (and hence spin).
                        \item $V$ admits a self-diffeomorphism $\theta: V \rightarrow V$, which acts as the identity on integral cohomology $H^*(V)$, but does not act as the identity on $K^*(V)$, where $K^*$ denotes complex $K$-theory.
                        \item If $n$ is odd, we can take $V$ to be a simply-connected rational homology sphere.
                    \end{enumerate}
                \end{PP}
                \begin{CC}
                    If $L$ is diffeomorphic to $V$, then $\theta$ does not lie in $\sG_L$.
                \end{CC}
                The rest of the section will be dedicated to a proof of Proposition \ref{Kexample}, following a suggestion of Randal-Williams (\cite{R-W}). We will construct these in sufficiently high odd dimensions and then observe that taking a product with $S^1$ provides the desired examples in all sufficiently high even dimensions.\par 
                Let $q$ be a positive integer, and $p$ an odd prime. Let $Y$ be the homotopy mapping cone of the map $S^{2q} \rightarrow S^{2q}$ of degree $p$. 
                \begin{LL}\label{Komputation}
                We have the following:
                    \begin{enumerate}
                        \item $$\tilde{K}^i(Y) \cong \begin{cases}
                            0 & \, i \, \mathrm{even}\\
                            \Z / p & \, i \, \mathrm{odd}
                        \end{cases}$$
                        \item When $i$ is odd,
                            $$K_i(Y) \cong 0$$
                        \item $$\tilde{H}^i(Y) \cong \begin{cases}
                            \Z / p & \, i = 2q + 1\\
                            0 & \, \mathrm{otherwise}
                        \end{cases}$$
                        \item $$\tilde{H}_i(Y) \cong \begin{cases}
                            \Z / p & \, i = 2q \\
                            0 & \, \mathrm{otherwise}
                        \end{cases}$$
                        \item For all $i$,
                            $$\tilde{H}^i(Y; \Q) = 0$$
                    \end{enumerate}
                \end{LL}
                \begin{proof}
                    Follows from the long exact sequence of a cone. (2) also uses the decomposition
                    $$K_i(Y) \cong \tilde{K}_i(Y) \oplus K_i(\mathrm{point})$$
                \end{proof}
                By \cite[Theorem 1.7]{AdamsJIV}, for sufficiently large $q$, there is a map
                $$g: \Sigma^{2(p - 1)} Y \rightarrow Y$$ 
                which induces an isomorphism on $\tilde{K}^*$, but must be 0 on $\tilde{H}^*$ and $\tilde{H}_*$ for degree reasons. Furthermore we can choose $q$ to be sufficiently large that $Y$ is simply connected.\par 
                Let $Z = \Sigma^{2(p - 1)} Y \vee Y$, and we define $f: Z \rightarrow Z$ to be the composition
                $$Z =\Sigma^{2(p - 1)} Y \vee Y \xrightarrow{\mathrm{Pinch} \vee Id} \Sigma^{2(p - 1)} Y \vee \Sigma^{2(p - 1)} Y \vee Y \xrightarrow{Id \vee g \vee Id} \Sigma^{2(p - 1)} Y \vee Y = Z$$
                where the first map is the pinching map, which uses the fact that 
                $$\Sigma^{2(p - 1)} Y = \Sigma \left( \Sigma^{2(p - 1) - 1} Y\right)$$
                and the second map sends the first copy of $\Sigma^{2(p-1)}Y$ to itself, the second copy of $\Sigma^{2(p-1)}Y$ to $Y$ (using $g$), and sends $Y$ to itself.
                \begin{LL} \label{reductions}
                    Let $E$ be a generalised homology or cohomology theory. Let $A$ be a connected finite CW complex, and $h: A \rightarrow A$ some map.\par 
                    Then $h$ acts as the identity on $E(A)$ if and only if $h$ acts as the identity on $\tilde{E}(A)$.
                \end{LL}
                \begin{proof}
                    There is a canonical splitting 
                    $$E(A) \cong \tilde{E}(A) \oplus E(\mathrm{point})$$
                    $h$ acts diagonally with respect to this splitting and always acts as the identity on $E(\mathrm{point})$.
                \end{proof}
                \begin{LL}\label{actionoff}
                    Let $f: Z \rightarrow Z$ be the map constructed above.
                    \begin{enumerate}
                        \item $f$ does not act as the identity on $K^*(Z)$.\par 
                        \item $f$ acts as the identity on $H^*(Z)$ and $H_*(Z)$. Therefore by Whitehead's and Hurewicz' theorems, since $Z$ is simply connected, $f$ is a homotopy equivalence.
                    \end{enumerate}
                \end{LL}
                \begin{proof}
                    With respect to the decomposition 
                    $$\tilde{K}^*(Z) \cong \tilde{K}^*\left(\Sigma^{2(p - 1)} Y\right) \oplus \tilde{K}^*(Y)$$
                    we see that $f^*$ is given by the block matrix
                    $$\begin{pmatrix}
                        Id & 0 \\
                        g^* & Id
                    \end{pmatrix}$$
                    and so $f$ does not act as the identity on $\tilde{K}^*(Z)$, and hence not on $K^*(Z)$ either, by Lemma \ref{actionoff}.
                    A similar argument shows the analogous results for $H^*(Z)$ and $H_*(Z)$.
                \end{proof}
                Now if $Z$ were a closed manifold with the right properties and $f$ were a diffeomorphism, we would be done, but unfortunately this is not the case.\par 
                Pick an embedding $Z \hookrightarrow \R^N$ for some large even $N > 4(p + q)$. Let $W$ be a regular neighbourhood of $Z$ with smooth boundary such that $W$ deformation retracts to $Z$. Then $W$ is a compact $N$-dimensional manifold with boundary, has one handle for each cell of $Z$, is simply connected, and for sufficiently large $N$, its boundary $\partial W$ is also simply connected. The tangent bundle $TW$ is trivial since $W$ is a codimension 0 submanifold of $\R^N$, and so $T\partial W$ is also stably trivial.\par 
                Since $Z \simeq W$, $f$ induces a map $f: W \rightarrow W$, well-defined up to homotopy. If $N$ was chosen to be sufficiently large, $f$ can be homotoped to an orientation-preserving embedding $e: W \hookrightarrow W$, which we can assume has image which does not touch $\partial W$.
                \begin{LL}
                    $e$ is homotopic to an orientation-preserving diffeomorphism $\phi: W \rightarrow W$.
                \end{LL}
                \begin{proof}
                    The complement $C := W \setminus e\left(W \setminus \partial W \right)$ is a cobordism from $e(\partial W)$ to $\partial W$. All three of $C$, $e(\partial W)$ and $\partial W$ are simply connected, and since $e$ is a homotopy equivalence, by excision (applied to the interior of $e(W)$) the inclusion $e(\partial W) \hookrightarrow C$ is a homology equivalence.\par 
                    By Poincar\'e-Lefschetz duality, the inclusion $\partial W \hookrightarrow C$ is also a homology equivalence, and so by the $h$-cobordism theorem $C$ is a trivial cobordism.\par 
                    We fix a trivialisation of $C$, meaning a diffeomorphism $C \cong \partial W \times [0, 1]_s$ sending $\partial W$ to $\partial W \times \{1\}$, where $s$ is the co-ordinate on $[0, 1]$. Extend $\partial_s$ to a vector field $P$ on the whole of $W$, and let $\rho$ be its time-1 flow.\par 
                    Then $\rho \circ e$ is isotopic to $e$ as an embedding, but is itself a diffeomorphism. Since $e$ was orientation-preserving, $\rho \circ e$ also is.
                \end{proof}
                Now we let $V = \partial W$ and $\theta = \phi|_{\partial W}$.
                \begin{LL}
                    $V$ is a rational homology sphere.
                \end{LL}
                \begin{proof}
                    By Lemma \ref{Komputation}, $Y$ and hence $Z$ and $W$ are rationally homology equivalent to a point. Then the result follows from the exact sequence of a pair (using Poincar\'e duality):
                    $$H_i(W; \Q) \rightarrow H^{N-i}(W; \Q) \rightarrow H_{i - 1}(V; \Q) \rightarrow H_{i - 1}(W; \Q) \rightarrow H^{N - i + 1}(W; \Q)$$
                \end{proof}
                The following two lemmas complete the proof of Proposition \ref{Kexample}.
                \begin{LL}
                    $\theta$ acts as the identity on $H^*(V)$.
                \end{LL}
                \begin{proof}
                    Note that since $\phi$ acts as the identity on $H_*(W)$ (by Lemma \ref{actionoff}), by Poincar\'e duality, $\phi$ also acts as the identity on $H^*(W, \partial W)$.\par 
                    From Lemma \ref{Komputation}, we see that
                    $$\tilde{H}^i(W) \cong \begin{cases}
                        \Z / p & \, i = 2q + 1 \, \mathrm{or}\, 2(p + q - 1) + 1\\
                        0 & \,\mathrm{otherwise}
                    \end{cases}$$
                    and since $W$ is orientable, by Poincar\'e duality
                    $$H^i(W, \partial W) \cong H_{N - i}(W) \cong \begin{cases}
                        \Z & \, i = N\\
                        \Z / p & \, i = N - 2q \, \mathrm{or} \, N - 2(p + q - 1)\\
                        0 & \, \mathrm{otherwise}
                    \end{cases}$$
                    Therefore by the long exact sequence of the pair $(W, \partial W)$, for all $i$ either the restriction map $\tilde{H}^i(W) \rightarrow \tilde{H}^i (\partial W)$ is an isomorphism or the boundary map $\tilde{H}^i(\partial W) \rightarrow H^{i + 1} (W, \partial W)$ is an isomorphism. Both of these maps are compatible with the actions of $\phi$ and $\theta$, so because $\phi$ acts as the identity on $H^*(W, \partial W)$ and on $H^*(W)$, the result follows.
                \end{proof}
                \begin{LL}
                    $\theta$ does not act as the identity on $K^*(V)$.
                \end{LL}
                \begin{proof}
                    We will show that the restriction map $\tilde{K}^i(W) \rightarrow \tilde{K}^i(V)$ is injective for all $i$, then the result follows from Lemmas \ref{actionoff} and \ref{reductions}.\par 
                    By the long exact sequence of a pair, the kernel of the restriction map is the image of 
                    $$K^i(W, \partial W) \rightarrow \tilde{K}^i(W)$$
                    $TW$ is trivial and hence oriented with respect to $K^*$, so by Atiyah duality
                    $$K^i(W, \partial W) \cong K_{N - i}(W)$$
                    Therefore using the decomposition 
                    $$\tilde{K}^*(W) \cong \tilde{K}^*\left(\Sigma^{2(p - 1)} Y\right) \oplus \tilde{K}^*(Y)$$ 
                    and Lemma \ref{Komputation}, we see that if $i$ is even then $\tilde{K}^i(W) = 0$, and if $i$ is odd then $K_{N - i}(W) = 0$. In either case this implies that the restriction map is injective in degree $i$.
                \end{proof}
        \section{An example in real $K$-Theory}\label{real K-theory Example}
            Let $U = S^3 \times S^2$. In this section we use the Hopf action of $S^3$ on $S^2$ to construct a simple diffeomorphism $U \rightarrow U$, and prove the following.
            \begin{PP}\label{KO example}
                There is a self-diffeomorphism $\zeta: U \rightarrow U$ which acts as the identity on integral homology but not as the identity on (the homology theory associated to) real $K$-theory $KO_*$.
            \end{PP}
            Then Theorem \ref{Main Theorem} implies
            \begin{CC} 
                If $L$ is diffeomorphic to $U$ and the condition of Proposition \ref{Orientation Criteria}(4) holds, then $\zeta$ does not lie in $\sG_L$.
            \end{CC}
            We view $S^3$ as the unit quaternions with unit $e \in S^3$, and we identify $S^2$ with the quotient $S^3 / S^1$ of $S^3$ by the right action of the unit complex numbers. Let 
            $$\eta: S^3 \rightarrow S^3/S^1 = S^2$$
            be the quotient map and let
            $$\mu: S^3 \times S^3 \rightarrow S^3$$
            be the product map. 
            $\mu$ descends to a map
            $$f: S^3 \times S^2 \rightarrow S^2$$
            which defines a left action of $S^3$ on $S^2$. We then define the diffeomorphism $\zeta$ to be the map $U \rightarrow U$ sending $(x, y)$ to
            $$(x, f(x, y)).$$
            Proposition \ref{KO example} then follows from Lemmas \ref{ yes id} and \ref{no id}.
            \begin{LL}\label{ yes id}
                $\zeta$ acts as the identity on the integral homology of $U$.
            \end{LL}
            \begin{proof}
                We first note that by the K\"unneth theorem,
                $$H_i(U) \cong \begin{cases}\Z & \,\mathrm{ if } \,i \in \{0, 2, 3, 5\}\\
                0 & \, \mathrm{ otherwise.}\end{cases}$$
                First note that $\zeta$ acts as the identity on $H_0(U)$ as $U$ is path-connected. \par 
                The following diagram commutes
                $$\begin{tikzcd}
                    \{e\} \times S^2 \arrow[d, hook, "\iota"] \arrow[dr, hook, "\iota"] & \\
                    U \arrow[r, "\zeta"] & U
                \end{tikzcd}$$
                where $\iota$ is the natural inclusion map. All maps in this diagram induce isomorphisms on $H_2$ so it follows that $\zeta$ acts as the identity on $H_2(U)$.\par  
                Similarly the following diagram commutes
                $$\begin{tikzcd}
                    U \arrow[r, "\zeta"] \arrow[d, "p"] & U \arrow[dl, "p"] \\
                    S^3
                \end{tikzcd}$$
                where $p: U \rightarrow S^3$ is the projection map onto the first coordinate. All maps in this diagram induce isomorphisms on $H_3$ so it follows that $\zeta$ acts as the identity on $H_3(U)$.\par 
                By the universal coefficients theorem, $\zeta$ acts as the identity on $H^2(U)$ and $H^3(U)$. Therefore since the cup product $H^2(U) \otimes H^3(U) \rightarrow H^5(U)$ is surjective, $\zeta$ acts as the identity on $H^5(U)$ and hence also on $H_5(U)$.
            \end{proof}
            \begin{LL}\label{no id}
                $\zeta$ does not act as the identity on $KO_3(U)$.
            \end{LL}
            \begin{proof}
                It follows from \cite{AtiyahK-TheoryandReality} that the map induced by $\eta$
                $$\eta_*: \Z \cong \tilde{KO}_3(S^3) \rightarrow \Z/2 \cong \tilde{KO}_3(S^2)$$
                is the quotient map $\Z \rightarrow \Z/2$.\par 
                Stably, $U$ is homotopy equivalent to a wedge product
                $$U \simeq \Sigma^2 \Sphere \vee \Sigma^3 \Sphere \vee \Sigma^5 \Sphere$$
                From this we see that
                \begin{align*}
                    \tilde{KO}_3(U) & \cong KO_3(\Sigma^2 \Sphere) \oplus KO_3(\Sigma^3 \Sphere) \oplus KO_3(\Sigma^5 \Sphere)\\
                    & \cong \Z \oplus \Z/2 \oplus 0
                \end{align*}
                and with respect to this decomposition, $\zeta_*$ is given by the matrix
                $$\begin{pmatrix}
                        Id_{\Z} & 0 \\
                        r & Id_{\Z/2}
                    \end{pmatrix}$$
                where $r: \Z \rightarrow \Z/2$ is the quotient map. This is not the identity so $\zeta$ does not act as the identity on $\tilde{KO}_3(U)$. By Lemma \ref{reductions}, $\zeta$ does not act as the identity on $KO_3(U)$.
            \end{proof}
            % \begin{proof}[Proof of Proposition \ref{KO example}]
            %     which we claim induces the identity on homology but not real $K$-theory.
            %     which induces an isomorphism
            %     $$\eta_*:\Z \cong \pi_3 S^3 \rightarrow \Z/2 \cong \pi_3 S^2$$
            %     and, stabilising, induces the quotient map on stable homotopy groups
            %     $$\eta_*: \Z \cong \tilde{\Sphere}_3 S^3 \rightarrow \Z/2 \cong \tilde{\Sphere}_3 S^2.$$
            %     Since the image of the Hopf element in $KO_1$ is non-zero, we see that the map
            %     $$\eta_* : \Z \cong \tilde{KO}_3 S^3 \rightarrow \Z/2 \cong \tilde{KO}_3 S^2$$
            %     is again the quotient map $\Z \rightarrow \Z/2$.\par 
            % \end{proof}
    \bibliography{Refs.bib}{} \bibliographystyle{abbrv}
    \Addresses
\end{document}